\documentclass[12pt]{amsart}
\usepackage{amssymb}
\usepackage{amsmath}
\usepackage{amscd}
\usepackage{amsthm}
\usepackage[centertags]{amsmath}
\usepackage{amsfonts}
\usepackage{newlfont}
\usepackage{graphicx}
\usepackage[usenames]{color}
\usepackage{amsfonts, amssymb}
\usepackage{mathrsfs}
\usepackage{latexsym}
\usepackage{tikz}\usepackage{verbatim}
\usepackage{tabulary,tabularx}
\usepackage[all]{xy}

\usepackage[latin1]{inputenc}

\newtheorem{theorem}{Theorem}[section]
\newtheorem{problem}[theorem]{Problem}

\newtheorem{proposition}[theorem]{Proposition}
\newtheorem{corollary}[theorem]{Corollary}
\newtheorem{lemma}[theorem]{Lemma}
\newtheorem{definition}[theorem]{Definition}
\theoremstyle{definition}
\newtheorem{remark}[theorem]{Remark}

\theoremstyle{definition}

\theoremstyle{remark}

\newcommand\Pp{\mathcal{P}}

\newcommand{\N}{\mathbb N}

\begin{document}
\title[The sup-norm vs. the norm of the coefficients]{The sup-norm vs. the norm of the coefficients: equivalence constants for homogeneous polynomials}

\author{Daniel Galicer \and Mart\'in Mansilla \and Santiago Muro}
\thanks{This work was partially supported by projects CONICET PIP 11220130100329CO,   ANPCyT PICT 2015-2224,  ANPCyT PICT 2015-2299, ANPCyT PICT 2015 - 3085, UBACyT 20020130100474BA, UBACyT 20020130300052BA, UBACyT 20020130300057BA. The second author was supported by a CONICET doctoral fellowship.}

\address{Daniel Galicer. Departamento de Matem\'{a}tica - Pab I,
	Facultad de Cs. Exactas y Naturales, Universidad de Buenos Aires
	(1428) Buenos Aires, Argentina and IMAS-CONICET} \email{dgalicer@dm.uba.ar} 

\address{Mart\'in Mansilla. Departamento de Matem\'{a}tica - Pab I,
	Facultad de Cs. Exactas y Naturales, Universidad de Buenos Aires
	(1428) Buenos Aires, Argentina and IMAS-CONICET} \email{mmansilla@dm.uba.ar}

\address{Santiago Muro. Departamento de Matem\'{a}tica - Pab I,
	Facultad de Cs. Exactas y Naturales, Universidad de Buenos Aires,
	(1428) Buenos Aires, Argentina and CIFASIS-CONICET} \email{muro@cifasis-conicet.gov.ar}


\begin{abstract}
Let $A_{p,r}^m(n)$ be the best constant that fulfills the following inequality: for every $m$-homogeneous polynomial $P(z) = \sum_{|\alpha|=m} a_{\alpha} z^{\alpha}$ in $n$ complex variables, $$\big( \sum_{|\alpha|=m} |a_{\alpha}|^{r} \big)^{1/r}  \leq  A_{p,r}^m(n) \sup_{z \in B_{\ell_p^n}} \big|P(z) \big| .$$
For every degree $m$, and a wide range of values of $ p,r \in [1,\infty]$ (including any $r$ in the case $p \in [1,2]$, and any $r$ and $p$ for the 2-homogeneous case), we give the correct asymptotic behavior of these constants as $n$ (the number of variables) tends to infinity.
Remarkably, in many cases, extremal polynomials for these inequalities are not (as traditionally expected) found using classical random unimodular polynomials, and special combinatorial configurations of monomials are needed. Namely, we show that Steiner polynomials (i.e., $m$-homogeneous polynomials such that the multi-indices corresponding to the nonzero coefficients form partial Steiner systems), do the work for certain range of values of $p,r$.

As a byproduct, we present some applications of these estimates to the interpolation of tensor products of Banach spaces, to the study of (mixed) unconditionality in spaces of polynomials and to the multivariable von Neumann's inequality.
\end{abstract}

\maketitle

\section{Introduction}
As usual we denote $\ell_{p}^{n}$ for the Banach space of all $n$-tuples $z=(z_1, \dots, z_n) \in \mathbb{C}^{n}$ endowed with the norm $\Vert (z_{1} , \ldots , z_{n}) \Vert_{p} = \Big( \sum_{i=1}^{n} \vert z_{i} \vert^{p} \Big)^{1/p}$ if $1 \leq p < \infty$, and $\Vert (z_{1} , \ldots , z_{n}) \Vert_{\infty} = \max_{i=1 , \ldots, n} \vert z_{i} \vert$ for $p = \infty$.  The unit ball of $\ell_{p}^{n}$ is denoted by $B_{\ell_p^n}$. For $1 \leq p \leq \infty$ we write $p'$ for its conjugate exponent (i.e., $\frac{1}{p}+\frac{1}{p'}=1$).

An $m$-homogeneous polynomial in $n$ variables is a function $P : \mathbb{C}^{n} \to \mathbb{C}$ of the form
\[
P(z_{1} , \ldots , z_{n}) = \sum_{\alpha \in \Lambda(m,n)}
a_{\alpha} z^{\alpha},
\]
 where $\Lambda(m,n):= \{ \alpha \in \mathbb{N}_0^n : |\alpha|:= \alpha_{1} + \cdots + \alpha_{n} = m \}$, $z^{\alpha}: = z_{1}^{\alpha_{1}} \cdots z_{n}^{\alpha_{n}}$ and $a_{\alpha} \in \mathbb{C}$.

Another way of writing a polynomial $P$ is as follows:
\[
P(z_{1} , \ldots , z_{n}) = \sum_{\substack{\mathbf j \in \mathcal J(m,n)}}
c_{\mathbf j} z_{\mathbf j},
\]
where  $\mathcal J(m,n):=\{ \mathbf j=(j_{1},\ldots , j_{k}) : 1 \leq j_{1} \leq \ldots \leq j_{k} \leq n \}$, $z_{\mathbf j}:=z_{j_{1}} \cdots z_{j_{k}}$ and $c_{\mathbf j} \in \mathbb{C}$. Note that $c_{\mathbf j}= a_{\alpha}$ with $\mathbf j=(1, \stackrel{\alpha_{1}}{\ldots}, 1, \ldots , n, \stackrel{\alpha_{m}}{\ldots} ,n)$.

We refer to the elements $(z^{\alpha})_{\alpha \in \Lambda(m,n)}$ (equivalently, $(z_{\mathbf j})_{\mathbf j \in \mathcal J(m,n)}$) as the monomials.

For $1 \leq p
\leq \infty$ we denote by $\mathcal{P} (^{m}\ell_{p}^{n} )$ the Banach space of all $m$-homogeneous
polynomials in $n$ complex variables equipped with the uniform (or sup) norm
\[
\Vert P \Vert_{\mathcal{P} (^{m}\ell_{p}^{n} )} := \sup_{z \in B_{\ell_p^n}} \big|P(z) \big|.
\]

Given an $m$-homogeneous polynomial $P(z) = \sum_{\alpha \in \Lambda(m,n)} a_{\alpha} z^{\alpha}$ in $n$  variables we denote the $\ell_r$-norm of its coefficients by $$\vert P \vert_r := \Big( \sum_{\alpha \in \Lambda(m,n)} |c_{\alpha}|^{r} \Big)^{1/r}.$$
Other norm, related to the coefficients is the so-called Bombieri $r$-norm defined in \cite{beauzamy1990products}:
$$[ P ]_r := \Big( \sum_{\alpha \in \Lambda(m,n)} \big(\frac{\alpha!}{m!}\big)^{r-1} |c_{\alpha}|^{r} \Big)^{1/r}.$$
The relation between the these coefficients-norms is given by the following inequalities (see \cite{beauzamy1990products}):
\begin{equation} \label{compbombieri}
(m!)^{\frac{1}{r}-1}  \vert P \vert_r  \leq [ P ]_r \leq \vert P \vert_r.
\end{equation}

In many contexts, it is essential to relate the summability of the coefficients of a given homogeneous polynomial with the sup-norm  on the unit ball of the ambient space.
This type of comparison results have shown several applications to different problems in complex and harmonic analysis, number theory and even in modern physics.
Among them we include the contributions to the study of: the asymptotic behavior of the Bohr radius \cite{defant2011bohnenblust,bayart2014bohr,defant2011bohr} or Dirichlet-Bohr radius \cite{carando2014dirichlet}, Sidon constants and convergence of Dirichlet series \cite{balasubramanian2006bohr,bayart2015monomial,bohnenblust1931absolute,defant2011bohnenblust,de2008ordre,konyagin2001translation,queffelec1995h}, monomial series expansions of holomorphic functions in infinitely many variables and multipliers of Dirichlet series \cite{bayart2014multipliers,bayart2015monomial}, classical inequalities in operator theory \cite{dixon1976neumann,mantero1979banach,galicer2015vonNeumann} and lower bounds on the classical bias obtainable in multiplayer XOR games \cite{montanaro2012some}.

Most of the applications mentioned above rely on the ingenious use of the celebrated Bohnenblust-Hille inequality (see \cite{bayart2014bohr,defant2011bohnenblust}). This inequality is a generalization to higher degrees of Littlewood's classical 4/3-inequality \cite{littlewood1930bounded} (a forerunner of Grothendieck's inequality)  and essentially bounds the $\ell_{\frac{2m}{m+1}}$-norm of the coefficients of an $m$-homogeneous polynomial in terms of its uniform norm on the polydisk. More precisely,
\begin{theorem}[Bohnenblust-Hille inequality] \label{BHineq}
There is a constant $C_{m,\infty}>0$ (which depends on $m$ but \emph{not on} $n$) such that, for every $m$-homogeneous polynomial $P$ in any number of complex variables $n$, we have:
\begin{equation}
\vert P \vert_{\frac{2m}{m+1}} \leq C_{m,\infty} \Vert P \Vert_{\mathcal{P} (^{m}\ell_{\infty}^{n} )}.
\end{equation}
\end{theorem}

Moreover, the $\ell_{\frac{2m}{m+1}}$-norm of the coefficients on the left hand side is optimal: i.e., there is no similar inequality replacing this norm by  other $\ell_{r}$-norm, for $r < \frac{2m}{m+1}$, involving a constant independent of the number of variables. \\
There are a number of analogues/generalizations of this inequality (which maintain the philosophy that the constant involved is totally independent of the number of variables), and consist in replacing the sup-norm on the ball of $\ell_\infty^n$ on the right hand side of the Bohnenblust-Hille inequality by other $\ell_p^n$-uniform norms (obviously changing the summability condition on the left side). These generalizations, inspired by some classical inequalities for bilinear forms \cite{hardy1934bilinear}, are known today in the literature as Hardy-Littlewood inequalities for homogenous polynomials and they have been carefully studied during the last years \cite{albuquerque2013optimal,dimant2013summation,praciano1981bounded}.
Precisely,

\begin{theorem}[Hardy-Littlewood type inequalities] \label{HLineq}
There is a constant $C_{m,p}>0$ (only depending on $m$ and $p$ and \emph{independent of} $n$) such that for every $m$-homogeneous polynomial in $n$-complex variables we have:
\begin{align*}
(i) & &\vert P \vert_{\frac{p}{p-m}} & \leq C_{m,p} \;\Vert P \Vert_{\mathcal{P}(^m\ell_p^n)} \hspace{0.5cm}  \text{ for } m \leq p \leq 2m, \\
(ii) & & \vert P \vert_{\frac{2mp}{mp+p-2m}} & \leq C_{m,p} \;\Vert P \Vert_{\mathcal{P}(^m\ell_p^n)} \hspace{0.5cm}  \text{ for } 2m \leq p.
\end{align*}
\end{theorem}

Again the exponents $\frac{p}{p-m}$ and $\frac{2mp}{mp+p-2m}$ in the above inequalities are the best possible.
Observe that, in the limit case ($p = \infty$) we recover the classical Bohnenblust-Hille exponent $\frac{2m}{m+1}$.

If we change any of the parameters involved on either or both sides of these Hardy-Littlewood inequalities, it is expected that the dependence on the number of variables becomes apparent. It is worth asking how this reliance is in terms of the summability of the coefficients, the uniform norm and the homogeneity degree considered. \\
Analogously, we can study a similar problem: the inequality that comes from exchanging the roles (sides of the inequality) between the norm of the coefficients and the uniform norm.

\begin{problem}\label{problema}
Let $A_{p,r}^m(n)$ and $B_{r,p}^m(n)$ be the smallest constants that fulfill the following inequalities: for every $m$-homogeneous polynomial $P$ in $n$ complex variables,
\begin{align*}
\vert P \vert_{r}  & \leq  A_{p,r}^m(n) \;\Vert P \Vert_{\mathcal{P}(^m\ell_p^n)}, \\
\Vert P \Vert_{\mathcal{P}(^m\ell_p^n)} & \leq B_{r,p}^m(n) \;\vert P \vert_{r}.
\end{align*}
How these constants behave in terms of the number of variables $n$? Which is their exact asymptotic growth?
\end{problem}

Observe that by \eqref{compbombieri}, the depence on $n$ of the constant that appears when comparing the sup-norm with the Bombieri norm is exactly the same as the constants related to Problem~\ref{problema}.

In the 80's, Goldberg \cite{goldberg1987equivalence} settled a similar problem in the context of matrix theory: given an $n \times n$ matrix $A$, he was interested in finding the best equivalence constant $c(r,p,n)$ (or its asymptotic behavior as $n$ tends to infinity) which relates the $\ell_r$-norm of the coefficients with the operator norm of $A$ acting on $\ell_p^n$. Partial and sharp results of this problem (and also some variants of it)  were given by Feng and Tonge in \cite{tonge2000equivalence,feng2003equivalence,feng2007equivalence}. Observe that Problem~\ref{problema} is essentially a polynomial version of Golberg's problem. Of course, this can also be settled for multilinear forms, whose constants turn out to have the same asymptotic growth. Indeed, note that using the notation of \cite{defant2009domains}, $A^m_{p,r}(n)=\|id:\otimes_{\varepsilon_s}^{m,s} \ell_{p'}^n \to \otimes_{\Delta_r}^{m,s} \ell_{r}^n\|$. In \cite[Proposition 3.1.]{defant2009domains} it is proved that
$$
\|id:\otimes_{\varepsilon_s}^{m,s} \ell_{p'}^n \to \otimes_{\Delta_r}^{m,s} \ell_{r}^n\|\sim \|id:\otimes_{\varepsilon}^{m} \ell_{p'}^n \to \otimes_{\Delta_r}^{m} \ell_{r}^n\|.
$$
This implies that the asymptotic behaviour  of the polynomial and the multilinear
cases are the same. 
 We should mention that some partial results for the multilinear problem were recently obtained by Araujo and Pellegrino \cite{araujo2015optimal} (see also \cite{botelho2010complex} for the case $p=2$).

For every degree $m$ and a wide range of values of $p,r \in [1,\infty]$, we give in Theorem~\ref{teorema con A} the correct asymptotic behavior of $A_{p,r}^m(n)$ as $n$ tends to infinity respectively. We also present in Proposition~\ref{asint B} the asymptotic growth of $B_{r,p}^m(n)$ for every $p,r \in [1,\infty]$. We also use these results to tackle two different problems: we present some applications of our estimates to the study of unconditionality in spaces of polynomials and the multivariable von Neumann's inequality.

\begin{definition} \label{defmix}
Let $(P_i)_{i \in \Lambda}$ be a Schauder basis of $\mathcal{P}(^m \mathbb{C}^n)$. For $ 1 \leq  p,q \leq \infty$ and $ n,m \in \mathbb{N}$ let $ \chi_{p,q}((P_i)_{i \in \Lambda})$ be the best constant $C > 0$ such that $$ \| \displaystyle\sum_{ i \in \Lambda} \theta_i c_i  P_i \|_{\mathcal{P}(^m \ell_q^n)} \leq C \| \displaystyle\sum_{ i \in \Lambda}  c_i P_i \|_{\mathcal{P}(^m \ell_p^n)},  $$ for every $P = \displaystyle\sum_{ i \in \Lambda}  c_i P_i \in \mathcal{P}(^m \mathbb{C}^n)$ and every choice of complex numbers $(\theta_i)_{i \in \Lambda}$ of modulus one.

The $(p,q)$-mixed unconditional constant of $\mathcal{P}(^m \mathbb{C}^n)$ is defined as
$$\chi_{p,q}(\mathcal{P}(^m \mathbb{C}^n)) := \inf\{\chi_{p,q}((P_i)_{i \in \Lambda}) :  (P_i)_{i \in \Lambda} \mbox{ basis for } \mathcal{P}(^m \mathbb{C}^n)\}.$$
\end{definition}

This notion was introduced by Defant, Maestre and Prengel in \cite[Section 5]{defant2009domains}.

In Section~\ref{secuncond} we provide correct estimates of the asymptotic growth of the mixed-$(p,q)$ unconditional constant as $n$ tends to infinity.
To achieve this we use some major results given in \cite{defant2009domains,bayart2015monomial} about domains of monomial convergence, our bounds on Problem \ref{problema} and multilinear interpolation. Moreover, we give an analog of a result of Pisier and Sch\"utt \cite{pisier1978some,schutt1978unconditionality} in this context. Namely, in order to study the asymptotic behavior of the mixed unconditional constants of $\mathcal{P}(^m \mathbb{C}^n)$, it is enough to look at the monomials $(z_{\mathbf j})_{\mathbf j \in \mathcal J(m,n)}$.
More precisely, we prove in Theorem~\ref{incondicionalidadbasemonomial} that $$\chi_{p,q}(\mathcal{P}(^m \mathbb{C}^n)) \sim \chi_{p,q}\big((z_{\mathbf j})_{\mathbf j \in \mathcal J(m,n)} \big).$$
We feel this result is interesting in its own right.

Mantero and Tonge considered in \cite{mantero1979banach} several versions of the multivariable von Neumann's inequality restricted to homogeneous polynomial on several commuting operators.
Among them, they were interested in the asymptotic behavior of the best possible constant $c(n)=c_{m,p,q}(n)$ such that
\begin{align*}
 \|P(T_1,\dots,T_n)\|_{\mathcal L (\mathcal H)} \le c(n) \|P\|_{\mathcal P(^m\ell_q^n)},
\end{align*}
for every $n$-tuple $T_1,\dots,T_n$  of commuting operators on a Hilbert space $\mathcal{H}$ satisfying
\begin{align}
\sum_{i=1}^n\|T_i\|_{\mathcal L (\mathcal H)}^p\le 1,
\end{align}
and every $m$-homogeneous polynomial $P$ in $n$-complex variables.
We use our estimates on Hardy-Littlewood type inequalities to obtain new upper bounds for the behavior of $c(n)$, and also address a related problem also treated in \cite{mantero1979banach}.

Before we present and prove our main results, we give some important comments. To show that our estimates are sharp it is essential to have polynomials with small norm with many nonzero coefficients. A widely used technique  to find these extremal polynomials is given by the probabilistic method: i.e., considering an $m$-homogeneous polynomial whose coefficients are given by independent random variables and computing it expectation (pretending to be small).
The systematic study of norms of random homogeneous polynomials arguably goes back to the classical Kahane-Salem-Zygmund theorem
\cite[Chapter 6]{kahane1993some}, which was found very useful in harmonic analysis. Recently,
additional applications of random (homogeneous) polynomials were given to complex and
functional analysis and to operator theory (see for example \cite{mantero1980schur,boas2000majorant,defant2001unconditional,defant2003bohr,defant2004maximum,carando2007extension,bayart2012maximum}).

Bayart in \cite{bayart2012maximum} (see also \cite{boas2000majorant,defant2003bohr,defant2004maximum}) exhibited polynomials with unimodular coefficients and with small sup-norm on the unit ball of $\ell_p^n$.
He showed that for each $1 \leq p \leq \infty$ there exists an $m$-homogeneous unimodular polynomial $P(z):=  \sum_{\alpha \in \Lambda(m,n)} \varepsilon_{\alpha} z^{\alpha}$ (i.e., $\varepsilon_{\alpha} = \pm 1$ for every $\alpha$) in $n$ complex variables  and a constant $K_{m,p}$ that depends  exclusively on $m$ and $p$ such that

\begin{equation} \label{polinomios de Bayart}
\Vert P \Vert_{\mathcal{P} (^{m}\ell_{p}^{n} )} \leq K_{m,p} \times
\begin{cases}
  n^{1- \frac{1}{p}} & \text{ if } 1 \leq p \leq 2, \\
  n^{m(\frac{1}{2} - \frac{1}{p}) + \frac{1}{2}} & \text{ if } 2 \leq p \leq \infty. \\
\end{cases}
\end{equation}

Observe that the number of non-zero coefficients is exactly the number of possible monomials, $\binom{n+m-1}{m}$. These polynomials will be very useful: they will be extremal in many ranges of values of $p,r \in [1,\infty]$ for the first inequality of Problem~\ref{problema}. Unfortunately, for a large range of values of $ p $ and $ r $ these polynomials become useless and new extremal examples are needed.
Therefore, it is important to relax the
number of terms appearing in the polynomials, by allowing them to have some zero coefficients, in order to reduce quantitatively the value of the sup-norm.
Obviously if one gets rid of many coefficients/monomials this helps considerably to lower the value of the norm but the important thing is to maintain an appropriate balance (having a sufficient number of non-zero coefficients but keeping the norm small).

We show that so-called Steiner polynomials, a special class of tetrahedral polynomials introduced by Dixon in \cite{dixon1976neumann} and studied in  \cite{galicer2015vonNeumann}  are accurate enough for our purposes.
Beyond the results exhibited in this article, this shows once again that Steiner polynomials play an important role in the area.  We believe that this motivates to study in depth these polynomials and its possible future applications.

We need some definitions to describe them. An $S_p(t, m, n)$ \textit{partial Steiner system} is a collection of subsets of
size $m$ of
$\{1,\dots,n\}$ such that every subset of $t$ elements is contained in at most one member of the
collection of subsets of size $m$.
 An $m$-homogeneous polynomial $P$ of $n$ variables  is a \textit{Steiner polynomial } if
there exists an $S_p(t, m, n)$ partial Steiner system $\mathcal S$ such that
$P(z_{1} , \ldots , z_{n}) = \sum_{\mathbf j \in \mathcal S} c_{\mathbf j} z_{\mathbf j}$ and $c_\mathbf j=\pm1$.
Note that the monomials involved in this class  have a particular combinatorial configuration.

The following result appears in \cite[Theorem 2.5.]{galicer2015vonNeumann}.

\begin{theorem}\label{Steiner con norma chica}
Let $m\ge 2$  and $\mathcal S$ be an $S_p(m-1,m,n)$ partial Steiner system. Then there exist signs
$(c_{\mathbf j})_{\mathbf j \in
\mathcal S}$ and a constant $D_{m,p}>0$ independent of $n$  such that the
$m$-homogeneous polynomial $P = \sum_{\mathbf j \in\mathcal S} c_{\mathbf j} z_{\mathbf j}$ satisfies
\[
\Vert P \Vert_{\mathcal{P}(^m\ell_p^n)} \leq D_{m,p} \times \begin{cases}
\log^{\frac{3p-3}{p}}(n) & \text{ for } 1\le p\le 2, \\
\log^{\frac{3}{p}}(n) n^{m(\frac{1}{2} - \frac{1}{p})} & \text{ for } 2 \leq p <\infty.
\end{cases}
\]
Moreover, the constant $D_{m,p}$ may be taken independent of $m$ for $p \neq 2$.
\end{theorem}

The last ingredient we need for the applications  is the existence of nearly
optimal partial Steiner systems, in the sense that they have many elements. This translates to many unimodular
coefficients of the Steiner polynomials. It is well known that any partial Steiner system $S_p(m-1,m,n)$ has
cardinality less than or equal to $\frac{1}{m} \binom{n}{m-1}$.

R\"odl \cite{Rod85} in the eighties proved that there exist partial Steiner systems $S_p(m-1,m,n)$ of cardinality at least $(1-o(1))\frac{1}{m}\binom{n}{m-1}$, where
$o(1)$ tends to zero as $n$ goes to infinity. Taking partial Steiner systems of this cardinality in Theorem~\ref{Steiner con norma chica} we have the
following.
\begin{corollary}\label{corosteiner}
Let $m\ge 2$. Then there exists a $m$-homogeneous Steiner
unimodular polynomial $P$ of $n$ complex variables with at least $C_m n^{m-1}$ unimodular coefficients satisfying the estimates in Theorem~\ref{Steiner con norma chica}, where $C_m$ is a constant that depends only on $m$.
\end{corollary}

 The article is organized as follows. In Section~\ref{Mainresult} we address Problem~\ref{problema}.
Section~\ref{secuncond} deals with the study of (mixed) unconditionality in spaces of polynomials. Finally, Section~\ref{applicvn} presents some applications to several versions of the multivariable von Neumann's inequality.

\section{Main results} \label{Mainresult}

If $(a_{n})_{n}$ and $(b_{n})_{n}$ are two sequences of real numbers we will write $a_{n} \ll b_{n}$ if there
exists a constant $C>0$ (independent of $n$) such that $a_{n} \leq C b_{n}$ for every $n$.
We will write $a_{n} \sim b_{n}$ if $a_{n} \ll b_{n}$ and $b_{n} \ll a_{n}$.
Recall that the number of $m$-homogeneous monomials in $n$ variables is $|\mathcal{J}(m,n)|=\binom {n + m -1} {m} \sim n^m$.

For every $P \in \mathcal{P}(^m \mathbb{C}^n)$ there exists a unique symmetric $m$-linear form $T$ such that for every $x \in \mathbb{C}^n$, $P(x)= T(x, \overset{m}{\ldots}, x)$ (see \cite{dineen1999complex}). We will denote the $r$-th coefficients norm of $T$ by $|T|_r$, that is,
$$
|T|_r := \Big( \displaystyle\sum_{ i \in \mathcal{M}(m,n)} \vert T(e_{i_1}, \ldots, e_{i_m}) \vert^r \Big)^\frac{1}{r},
$$
where  $\mathcal{M}(m,n)=\{\mathbf i=(i_{1},\ldots ,i_{m}) \,:\, 1 \leq i_{l} \leq n, \, 1\le l\le m \}$.
It is well known that there exist constants $ C_l=C_l(m)>0$, $l=1,2,$ independent of $n$, such that for every $P \in \mathcal{P}(^m \mathbb{C}^n)$ and its associated symmetric $m$-linear form $T$ we have
\begin{align}
 \vert T \vert_r & \leq  \vert P \vert_r   \leq  C_1  \vert T \vert_r & \text{ for } 1 \leq r \leq \infty, \label{P vs T coef}\\
C_2 \| T \|_{\mathcal{L}(^m \ell_p^n)} & \leq  \Vert P \Vert_{\mathcal{P}(^m \ell_p^n)}   \leq   \Vert T \Vert_{\mathcal{L}(^m \ell_p^n)} & \text{ for } 1 \leq p \leq \infty. \label{P vs T unif}
\end{align}

We now state our main theorem.

\begin{theorem} \label{teorema con A}
Let $A_{p,r}^m(n)$ be the smallest constant such that for every $m$-homogeneous polynomial $P$ in $n$ complex variables, $$  \vert P \vert_r  \leq  A_{p,r}^m(n) \;\Vert P \Vert_{\mathcal{P}(^m\ell_p^n)}.$$
Then,
\[
\begin{cases}  \; A_{p,r}^m(n) \sim 1 & \text{ for } \; (A): \; [\frac{1}{2}\le\frac1{r}\le \frac{m+1}{2m}-\frac1{p} ] \text{ or } [\frac{1}{r} \leq \frac{1}{2} \; \wedge \; \frac{m}{p} \leq 1 - \frac{1}{r} ], \\

\;A_{p,r}^m(n) \sim n^{\frac{m}{p}+\frac1{r}-1} & \text{ for } \; (B): \; [\frac{1}{2m} \leq \frac{1}{p} \leq \frac{1}{m} \; \wedge \; -\frac{m}{p} + 1 \leq \frac{1}{r} ], \\

\;A_{p,r}^m(n) \sim n^{m (\frac{1}{p}+\frac{1}{r}-\frac{1}{2}) - \frac{1}{2}} & \text{ for } \; (C): \; [ \frac{m+1}{2m} \leq \frac{1}{r} \; \wedge \; \frac{1}{p} \leq \frac{1}{2} ]  \text{ or } [ \frac{1}{2} \leq \frac{1}{r} \leq \frac{m+1}{2m} \leq \frac{1}{p} + \frac{1}{r} \wedge \; \frac{1}{p} \leq \frac{1}{2}], \\
\;A_{p,r}^m(n) \sim n^{\frac{m}{r} + \frac1{p}  -1} & \text{ for } \; (D): \; [ \frac{1}{2} \leq \frac{1}{p} \; \wedge \; 1-\frac{1}{p} \leq \frac{1}{r} ], \\
\;A_{p,r}^m(n) \ll n^{\frac{m-1}{r}} & \text{ for } \; (E): \; [ \frac{1}{2} \leq \frac{1}{p} \leq 1 - \frac{1}{r} ], \\

\;A_{p,r}^m(n) \sim n^{\frac{1}{r}} & \text{ for } \; (F): \; [ \frac{m-1}{p} \leq 1 - \frac{1}{r} \; \wedge \; \frac{1}{m} \leq \frac{1}{p} \leq \frac{1}{m-1} ], \\

\end{cases}
\]
Moreover, the power of $n$ in $(E)$ cannot be improved.
\end{theorem}

The first figure represents the regions described in Theorem~\ref{teorema con A}. For the blank region we do not know right order of $A_{p,r}^m(n)$ (see the comments after  Remark~\ref{resp interp compleja} below).
It is noteworthy that much of the work is to determine \emph{which are the regions} to consider.

\begin{figure}\label{figure H-L1}
\begin{center}
\begin{tikzpicture}[scale=1.15]

	\draw (0.6,2.4) node {$(A)$};
  \path[draw, shade, top color=yellow, bottom color=yellow, opacity=.3]
     (1,4) node[below] {$ $}  -- (2, 0) -- (0, 0) -- (0,5) -- cycle;

	\draw (2,6) node {$(C)$};
  \path[draw, shade, left color=yellow, right color=red, opacity=0.3]
    (0,5) node[left] {$ $} -- (1,4) node[below] {$ $}
    -- (4, 4) -- (4,8) -- (0, 8) -- cycle;

	\draw (6,5) node {$(D)$};
  \path[draw, shade, bottom color=yellow, top color=blue, opacity=.3]
    (4,8) node[below] {$ $} -- (8,8) node[below] {$ $}
    -- (8, 0) -- (4, 4) -- cycle;

	\draw (5.5,1.5) node {$(E)$};
  \path[draw, shade, top color=blue, bottom color=yellow, opacity=.3]
    (4,4) node[below] {$ $} -- (8,0) node[below] {$ $}
     -- (4, 0) -- cycle;

	\draw (1.6,3) node {${(B)}$};
  \path[draw, shade, top color=violet, bottom color=yellow, opacity=.3]
    (2,4) node[left] {$ $} -- (1,4) node[below] {$ $}
     -- (2, 0)  -- cycle;

	\draw (2.25,0.4) node {${(F)}$};
  \path[draw, shade, left color=green, bottom color=blue, opacity=.3]
    (2,2) node[left] {$ $} -- (2.66,0) node[below] {$ $}
     -- (2, 0)  -- cycle;

\draw[dotted] (2,2) -- (0,2); 
\draw[dotted] (0,4) -- (8,4);
\draw[dotted] (1,4) -- (1,0);

\draw (0,4) node[left] {$\frac12$};
\draw (4,0) node[below] {$\frac12$};
\draw (2.66,0) node[below] {$\frac1{m-1}$};
\draw (2, 0) node[below] {$\frac1{m} $};
\draw (0,2) node[left] {$\frac1{m} $};
\draw (1, 0) node[below] {$\frac1{2m} $};
\draw (0,5) node[left] {$\frac{m+1}{2m}$};

\draw (8.1, 0) node[below] {$\frac1{p} $};
  \draw[->] (-0.3,0) -- (8.3, 0);
\draw (0,8.1) node[left] {$\frac1{r} $};
  \draw[->] (0,-0.3) -- (0, 8.3);
 \end{tikzpicture}
\end{center}

\caption{Graphical overview of the regions described in Theorem~\ref{teorema con A}. }

\end{figure}
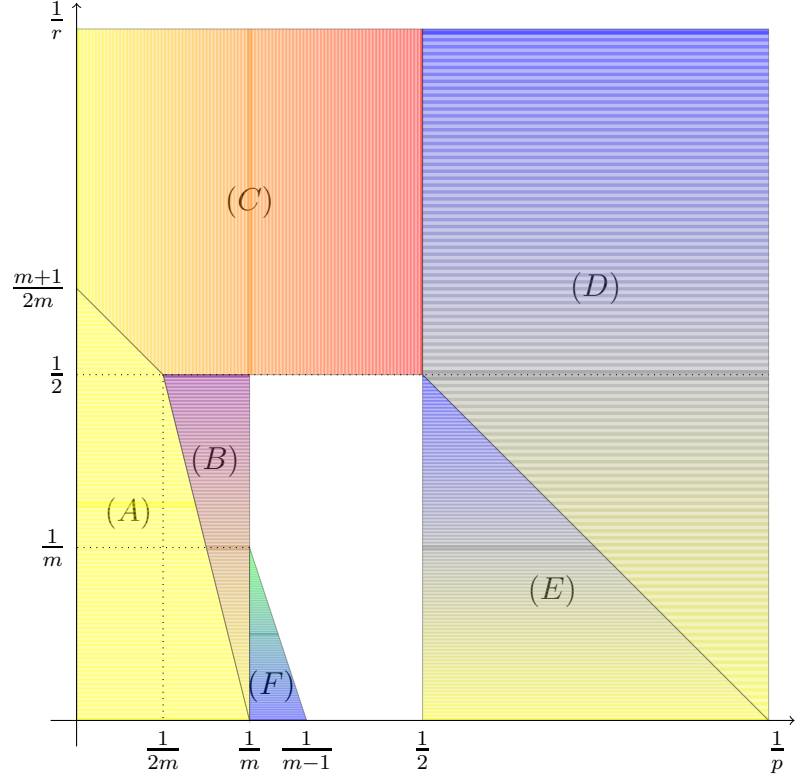

Also note that for $m=2$ we have a complete description of the asymptotics of $A_{p,r}^2 (n)$. For $r\ge 2$ this can also be deduced as a consequence of \cite[Theorems 1 and 2]{feng2007equivalence}. For $r<2$ the results are, up to our knowledge, new.

\begin{proof}
Let $P$ be an $m$-homogeneous polynomial in $n$ complex variables and $T$ its associated symmetric $m$-linear form. We will use several times the following inequalities, which are a simple consequence of H\"older's inequality,
\begin{align}
\vert P \vert_q  \leq \binom {n + m -1} {m}^{\frac{1}{q} - \frac{1}{r}} \vert P \vert_r & \ll n^{m(\frac{1}{q} - \frac{1}{r})} \vert P \vert_r &\text{ if }& 1\leq q \leq r \leq \infty. \label{comparacion coef}\\
\Vert P \Vert_{\mathcal{P}(^m\ell_p^n)}   & \leq n^{m(\frac{1}{q} - \frac{1}{p})} \Vert P \Vert_{\mathcal{P}(^m\ell_q^n)} &\text{ if }& 1\leq q \leq p \leq \infty. \label{comparacion unif}
\end{align}

Case $(A)$ is immediate from the Hardy-Littlewood inequalities and cases $(B)$ and $(C)$ also appear in \cite{araujo2015optimal}.

$\bullet (A):$  Suppose first that $\frac{1}{2}\le\frac1{r}\le \frac{m+1}{2m}-\frac1{p}$. If $q:=\frac{2mr}{(m+1)r-2m}$ then $2m  \leq q \leq p$ and by the Hardy-Littlewood inequality, Theorem~\ref{HLineq} $(ii)$,
we have $$\vert P \vert_r \ll \Vert P \Vert_{\mathcal{P}(^m\ell_q^n)} \ll \Vert P \Vert_{\mathcal{P}(^m\ell_p^n)}.$$
Now suppose $\frac{1}{r} \leq \frac{1}{2}$ and $\frac{m}{p} \leq 1 - \frac{1}{r}$. If we set $q:=\frac{mr}{r-1}$ then $m \leq q \leq \min\{p,2m\}$; then reasoning as before (but  using part $(i)$ of Theorem~\ref{HLineq}) we can easily reach the same conclusion.

$\bullet (B):$ Taking $ p \leq q = \frac{mr}{r-1} $, by the Hardy-Littlewood inequality, Theorem~\ref{HLineq} $(i)$, and \eqref{comparacion unif} it follows $$ | P |_r \ll \| P \|_{\mathcal{P}(^m \ell_q^n)} \ll \| P \|_{\mathcal{P}(^m \ell_p^n)} n^{m(\frac{1}{p} - \frac{1}{q})} = \| P \|_{\mathcal{P}(^m \ell_p^n)} n^{\frac{m}{p} + \frac{1}{r} - 1 } .$$

For the optimality we can take the polynomial $$ P = \displaystyle\sum_{j=0}^{k-1} z_{m j + 1} \cdots z_{mj + m}, \; \; \mbox{ with } \; \; k = \Big[\frac{n}{m} \Big],$$ it can be seen using Lagrange multipliers and the fact that $ p \geq m $, that $$ \| P \|_{\mathcal{P}(^m \ell_p^n)} = k \Big(\frac{1}{mk} \Big)^\frac{m}{p} \sim n^{1-\frac{m}{p}}.$$
Then,
$$
n^\frac{1}{r} \sim   k^\frac{1}{r}=| P |_r \leq  A_{p,r}^m (n) \| P \|_{\mathcal{P}(^m \ell_p^n)}
 \sim A_{p,r}^m(n) n^{1-\frac{m}{p}},
 $$
 and therefore $n^{\frac{m}{p} + \frac{1}{r} - 1 }\ll A_{p,r}^m.$

$\bullet  (C):$ Suppose $\frac{m+1}{2m} \leq \frac{1}{r}$ and $\frac{1}{p} \leq \frac{1}{2}$. Using the Bohnenblust-Hille  inequality Theorem~\ref{BHineq}, inequalities \eqref{comparacion coef} and \eqref{comparacion unif} we have
\begin{align*}
\vert P \vert_r & \ll n^{m(\frac{1}{r} - \frac{m+1}{2m})} \vert P \vert_{\frac{2m}{m+1}} \ll n^{m(\frac{1}{r} - \frac{m+1}{2m})} \Vert P \Vert_{\mathcal{P}(^m\ell_{\infty}^n)} \\
&\ll n^{m(\frac{1}{r} - \frac{m+1}{2m})} n^{\frac{m}{p}} \Vert P \Vert_{\mathcal{P}(^m\ell_p^n)} = n^{m (\frac{1}{p}+\frac{1}{r}-\frac{1}{2}) - \frac{1}{2}} \Vert P \Vert_{\mathcal{P}(^m\ell_p^n)}.
\end{align*}

Suppose  $\frac{1}{2} \leq \frac{1}{r} \leq \frac{m+1}{2m} \leq \frac{1}{p} + \frac{1}{r}$ and let $q:= \frac{2mr}{(m+1)r-2m}$.
Note that $\max\{2m,p\}\leq q$. By the Hardy-Littlewood inequality, Theorem~\ref{HLineq} $(ii)$ and \eqref{comparacion unif} we get

$$ \vert P \vert_r \ll \Vert P \Vert_{\mathcal{P}(^m\ell_q^n)} \ll  n^{m(\frac{1}{p}-\frac{1}{q})}\Vert P \Vert_{\mathcal{P}(^m\ell_p^n)} = n^{m (\frac{1}{p}+\frac{1}{r}-\frac{1}{2}) - \frac{1}{2}} \Vert P \Vert_{\mathcal{P}(^m\ell_p^n)}.$$

To show that this asymptotic growth is optimal, we consider $P$ an $m$-homogeneous unimodular polynomial as in \eqref{polinomios de Bayart}. Then, as $ \frac{1}{p} \leq \frac{1}{2}$

$$n^{\frac{m}{r}} \ll \vert P \vert_r \leq  A_{p,r}^m(n)  \Vert P \Vert_{\mathcal{P}(^m\ell_p^n)} \ll A_{p,r}^m(n) n^{m(\frac{1}{2} - \frac{1}{p}) + \frac{1}{2}}.$$
Therefore, $$n^{m (\frac{1}{p}+\frac{1}{r}-\frac{1}{2}) - \frac{1}{2}} = n^{\frac{m}{r}-  [m(\frac{1}{2} - \frac{1}{p}) + \frac{1}{2}]} \ll A_{p,r}^m(n).$$

$\bullet (D):$ We adapt an argument based on \cite{cobos1999gp} used by the authors in the case $p=r=2$ to show the relationship between the Hilbert-Schmidt and supremum norm of a multilinear form on $\ell_2^n$. If $T$ is the symmetric $m$-linear form associated to $P$, it induces a $(m-1)$-linear mapping $ \tilde{T} \in \mathcal{L}(^{m-1} l_p^n; (l_p^n)^*)$, defined by
$$
\tilde{T}(x_1, \ldots,x_{m-1})(\cdot) = T(x_1, \ldots, x_{m-1}, \cdot).
$$
Then
\begin{eqnarray*}
| T |_r^r & = & \displaystyle\sum_{ i \in \mathcal{M}(m,n)} \vert T(e_{i_1}, \ldots, e_{i_m}) \vert^r \\
& = & \displaystyle\sum_{ i \in \mathcal{M}(m-1,n)} \displaystyle\sum_{l=1}^n \vert T(e_{i_1}, \ldots, e_{i_{m-1}}, e_l) \vert^r \\
& \overset{\mbox{H\"older}}{\leq} & \displaystyle\sum_{i \in \mathcal{M}(m-1,n)} (\displaystyle\sum_{l=1}^n \vert T(e_{i_1}, \ldots, e_{i_{m-1}}, e_l) \vert^{p'} )^\frac{r}{p'} n^{\frac{r}{p} +1-r}\\
& \leq & n^{m-1 + \frac{r}{p} +1-r} \displaystyle\sup_{ \Vert x_i \Vert_{p} \leq 1} \Vert \tilde{T}(x_1, \ldots, x_{m-1}) \Vert_{p'}^r \\
& = &  n^{m + \frac{r}{p} - r} \Vert T \Vert_{\mathcal{L}(^m \ell_p^n)}^r, \\
\end{eqnarray*}
where in the first inequality we used H\"older inequality in the case $ \frac{p'}{r} \geq 1$.
Then by equations \eqref{P vs T coef} and \eqref{P vs T unif} we have $$ |P|_r \ll n^{\frac{m}{r} + \frac{1}{p} - 1} \| P \|_{\ell_p^n}.$$

For the optimality, we use \eqref{polinomios de Bayart} so, since $ 1 \leq p \leq 2 $ there exists a unimodular polynomial $P$ such that
$$ n^{\frac{m}{r}} \ll | P |_r \leq A_{p,r}^m(n) \| P \|_{\mathcal{P}(\ell_p^n)} \ll A_{p,r}^m (n) n^{1 - \frac{1}{p}}.$$

$\bullet (E):$ Observe that $$ A_{p,r}^m(n) = \| id : \mathcal{P}(^m \ell_p^n) \rightarrow (\mathcal{P}(^m \mathbb{C}^n), | \cdot |_r) \|.$$
Thus, if $ \frac{1}{r} = \frac{\theta}{p'} $, for $0<\theta<1$, we can apply complex interpolation to conclude that
$$ A_{p,r}^m(n) \leq (A_{p,p'}^m(n))^\theta (A_{p,\infty}^m(n))^{1 - \theta}.$$
Since $1\le p\le 2$, we have by part $(D)$ that $A_{p,p'}^m(n) \sim n^{\frac{m-1}{p'}}$ and also, applying the Cauchy integral formula we deduce that $A_{p,\infty}^m(n) \sim 1$. Therefore, we obtain $$A_{p,r}^m(n) \ll n^{ \frac{m-1}{r}}.$$

For the lower bound, taking a Steiner polynomial $P \in \mathcal{P}(^m \mathbb{C}^n)$ as in Corollary~\ref{corosteiner} whose associated partial Steiner system has cardinality $\gg n^{m-1}$ and $1\le p\le 2$, then
$$ n^{\frac{m-1}{r}} \ll |P|_r \leq A_{p,r}^m (n) \| P \|_{\mathcal{P}(^m \ell_p^n)}
 \ll A_{p,r}^m (n) \log^{\frac{3p-3}{p}}(n).$$

 Hence, we have that for every $\varepsilon > 0,$ $$ n^{ \frac{m-1}{r} - \varepsilon} \ll A_{p,r}^m (n).$$

$\bullet (F):$ Let $T$ be the symmetric $m$-linear form associated to $P$ and, given $ 1 \leq i \leq n$, let us define $ T_i \in \mathcal{L}(^{m-1} \mathbb{C}^n)$ as
$$
T_i(x_2, \ldots, x_m) = T(e_i, x_2, \ldots, x_m) .
$$
Then
\begin{eqnarray}
|P|_r^r\sim |T|_r^r & = & \displaystyle\sum_{ i \in \mathcal{M}(m,n)} \vert T(e_{i_1}, \ldots, e_{i_m}) \vert^r \nonumber \\
& = & \displaystyle\sum_{i=1}^n | T_i |_r^r  \nonumber \\
& \ll & \displaystyle\sum_{i=1}^n \| T_i \|_{\mathcal{L}(^{m-1} \ell_p^n)}^r \label{recurr m-1} \\
& \leq & n \| T \|_{\mathcal{L}(^m \ell_p^n )}^r \sim n \|P\|_{\mathcal{P}(^m \ell_p^n )}^r \nonumber ,
\end{eqnarray}
where we have used in \eqref{recurr m-1} the fact that $A_{p,r}^{m-1}(n)\sim 1$ for this range of $p$ and $r$.
 Therefore $$ | P |_r \ll n^{\frac{1}{r}} \| P \|_{\mathcal{P}(^m \ell_p^n)}.$$

For the lower bound, let $P= \displaystyle\sum_{j=1}^k z_{m j + 1} \cdots z_{mj + m}$ as in part $(B)$, then since $p\ge m$ (in region $(F)$), we have that $\| P \|_{\mathcal{P}(^m \ell_p^n)}\sim 1$ and thus
$$
n^{\frac{1}{r}} \sim | P |_r \ll   A_{p,r}^m (n) \| P \|_{\mathcal{P}(^m \ell_p^n)} \sim  A_{p,r}^m (n).
$$

\end{proof}

For $2\le p \le m$, $2\le r<\infty$ and $(\frac1{p},\frac1{r})\notin (F)$ we could have used interpolation (in vertical direction, as we did in the proof of part $(E)$ of Theorem~\ref{teorema con A}) to obtain effective upper bounds for $A_{p,r}^m $. We choose not to state them explicitly since  we  believe  these  estimates  are  suboptimal.

\subsection{An application to interpolation of tensor products of Banach spaces}

Kouba in \cite{kouba1991interpolation} (see also \cite{defant2003interpolation,defant2000complex}) proved a remarkable result on complex interpolation of injective tensor products of Banach spaces, which implies that for $L_p$-spaces,
\begin{equation*}
[L_{p_0}\otimes_\varepsilon L_{q_0},L_{p_1}\otimes_\varepsilon L_{q_1}]_\theta = L_{p}\otimes_\varepsilon L_{q},
\end{equation*}
for $0<\theta<1$, $1\le p_0,q_0,p_1,q_1\le 2$, $\frac1{p}=\frac{1-\theta}{p_0}+\frac{\theta}{p_1}$ and $\frac1{q}=\frac{1-\theta}{q_0}+\frac{\theta}{q_1}$.

We will show that  Theorem \ref{teorema con A} implies that a similar statement does not hold for the $m$-fold injective tensor, $m>2$.
Indeed, we will show that the following problem has a negative answer.
\begin{problem}[]\label{interp compleja}
Let $1 \leq p_0,p_1 \leq 2 $, $ 0 < \theta < 1$ and $n,m \in \mathbb{N}$ and consider
$$\iota_{\theta} : \otimes_{i=1, \varepsilon}^m [\ell_{p_0}^n, \ell_{p_1}^n]_{\theta} \longrightarrow [\otimes_{i=1, \varepsilon}^m \ell_{p_0}^n, \otimes_{i=1, \varepsilon}^m \ell_{p_1}^n ]_{\theta},$$ the natural inclusion.
Is there any constant $ C> 0 $ independent of $n$ such that $$ \| \iota_\theta \|\leq C ?$$
\end{problem}

We are sincerely grateful to Jorge Tom\'as Rodriguez, who gently gave us the idea of the following family of counterexamples.

\begin{remark}\label{resp interp compleja}
The answer to the question in Problem \ref{interp compleja} is negative in general. In particular, for $m \ge 3$, $r > m$, $ (m-1) r' \le p_1' < m$, $ m r' \le p_0' $ and $\theta$ such that $\frac{1}{p'} = \frac{1-\theta}{p_0'} + \frac{\theta}{p_1'} \le m $ there is no $C>0$ independent of $n$ such that $\| \iota_\theta \| \le C$.
\end{remark}

\begin{proof}
Observe that $(\frac1{p_0'},\frac1{r})\in (A)$ and  $(\frac1{p_1'},\frac1{r})\in (F)$ with $p_1' > m$. Assuming a positive answer to Problem~\ref{interp compleja}, we would have that for, $\frac1{p'}=\frac{1-\theta}{p_0'}+\frac{\theta}{p_1'}$, $\theta\in (0,1)$, 
 $$
 A_{p,r}^m(n)\le n^{\frac{1-\theta}{r}}.
 $$
If $\theta$ fulfills the condition $\frac{1}{p'} = \frac{1-\theta}{p_0} + \frac{\theta}{p_1} \le m $ we have $(p',r) \in (F)$, this contradicts the lower bound from region $(F)$ in Theorem \ref{teorema con A} .
\end{proof}

In \cite{bayart2018coincidence} the authors gave a similar result independently and with a different approach. Observe that the question in Problem \ref{interp compleja} still unanswered when both $ 2 \le p_0',p_1' \le m  $ and $p_0' \neq p_1'$. In the case the answer is affirmative for these parameters, this would allow us to prove the remaining cases in Theorem \ref{teorema con A}.

\subsection{Asymptotic estimates for $B_{r,p}^m (n)$.} We now present the correct asymptotic behavior for the constants $B_{r,p}^m (n)$ defined in Problem~\ref{problema}. These estimates will be useful in the next section for the applications.

\begin{proposition}\label{asint B}
Let $B_{r,p}^m(n)$ be the smallest constant such that for every $m$-homogeneous polynomial $P$ in $n$ complex variables, $
\Vert P \Vert_{\mathcal{P}(^m\ell_p^n)} \leq B_{r,p}^m(n) \;\vert P \vert_{r}.$
We have
$$ B_{r,p}^m(n) \sim
\begin{cases}
1 & \text{ for } r \leq p', \\
n^{m (1-\frac{1}{p}-\frac{1}{r})} & \text{ for } r \geq p'.
\end{cases}
$$

\end{proposition}

\begin{proof}
Let $n,m \in \mathbb{N}$ and $1 \leq p, r \leq \infty$.
Let $ P = \displaystyle\sum_{\alpha \in \Lambda(m,n) } a_{\alpha} z^{\alpha} $ be an $m$-homogeneous polynomial in $n$ variables.
Suppose first that $r \leq p'$. Then
\begin{eqnarray*}
\Vert P \Vert_{\mathcal{P}(^m \ell_p^n)} & = & \displaystyle\sup_{z \in B_{\ell_p^n}} \vert \displaystyle\sum_{\alpha \in \Lambda(m,n)} a_{\alpha} z^{\alpha} \vert \\
& \leq & \displaystyle\sup_{z \in B_{\ell_p^n}} (\displaystyle\sum_{\alpha \in \Lambda(m,n)} \vert a_{\alpha} \vert^{p'} )^{\frac{1}{p'}} \displaystyle(\sum_{\alpha \in \Lambda(m,n)} \vert z^{\alpha} \vert^{p} )^{\frac{1}{p}} \\
& \leq & | P |_{p'} \displaystyle\sup_{z \in B_{\ell_p^n}}(\displaystyle\sum_{\mathbf{i} \in \mathcal M(m,n)} \vert z_\mathbf{i} \vert^{p} )^{\frac{1}{p}} \\
& = &  | P |_{p'} \displaystyle\sup_{z \in B_{\ell_p^n}}(\displaystyle\sum_{k=1}^n \vert z_k \vert^{p})^{\frac{m}{p}} \\
& = &  | P |_{p'} \\
&\le& | P |_{r}.
\end{eqnarray*}

On the other hand, if $ r \geq p'$,
\begin{eqnarray*}
\Vert P \Vert_{\mathcal{P}(^m \ell_p^n)} & \leq & | P |_{p'} \\
& \leq & | P |_r n^{m(\frac{1}{p'} - \frac{1}{r})} \\
& = & | P |_r n^{m(1  - \frac{1}{p}- \frac{1}{r})}.
\end{eqnarray*}

To study lower bounds, let us take the polynomial $P(z) = \sum_{\mathbf{j} \in \mathcal J(m,n)} z_{\mathbf{j}}$. Note that $ | P  |_r \sim n^{\frac{m}{r}} $ and
\begin{eqnarray*}
\Vert P \Vert_{\mathcal{P}(^m \ell_p^n)} & = & \displaystyle\sup_{z \in B_{\ell_p^n}} \vert \displaystyle\sum_{\mathbf{j} \in \mathcal J(m,n)} z_{\mathbf{j}} \vert \\
& \geq & \vert \displaystyle\sum_{\mathbf{j} \in \mathcal J(m,n)} n^{-\frac{m}{p}} \vert \; \; \; \; \;\mbox{  taking  } z=(\overbrace{ \frac{1}{n^{\frac{1}{p}}}, \ldots, \frac{1}{n^{\frac{1}{p}}}}^m) \\
& \sim & n^{m(1-\frac{1}{p})}.\\
\end{eqnarray*}

Therefore $B_{r,p}^m(n) \gg n^{m(1 - \frac{1}{r} - \frac{1}{p})}.$
\end{proof}

\section{Mixed unconditional basis constant for homogeneous polynomials on $\ell_p$ spaces} \label{secuncond}
 Here we will study the asymptotic growth of $\chi_{p,q}(\mathcal{P}(^m \mathbb{C}^n))$ for fixed $1 \leq q,p \leq \infty$ and $m \in \mathbb{N}$ as $n$ tends to infinity (see Definition~\ref{defmix}).

The following result shows that, in order to study the asymptotic behavior of the mixed unconditional constants of $\mathcal{P}(^m \mathbb{C}^n)$, it is enough to understand what happens with the monomial basis $(z_{\mathbf j})_{\mathbf j \in \mathcal{J}(m,n)}$. These can be seen as a sort extension of a result of Pisier and Sch\"utt \cite{pisier1978some,schutt1978unconditionality} (see also \cite{defant2001unconditional, defant2011bohr,carando2011unconditionality}).

\begin{theorem} \label{incondicionalidadbasemonomial}
We have the following relation:
$$  \chi_{p,q}(\mathcal{P}(^m \mathbb{C}^n)) \leq \chi_{p,q}\big((z_{\mathbf j})_{\mathbf j \in \mathcal{J}(m,n)}\big) \leq 2^m \chi_{p,q}(\mathcal{P}(^m \mathbb{C}^n)).$$
\end{theorem}

Our proof relies on Szarek's approach \cite{szarek1981note} combined with the following inequality due to Bayart \cite{bayart2002hardy} (see also \cite{weissler1980logarithmic,defant2015p}).

\begin{lemma}[Bayart's inequality]
Let $P(z)=\sum_{\mathbf j\in\mathcal J(m,n)}c_{\mathbf j} z_{\mathbf j}$ be an $m$-homogeneous polynomial in $n$-variables. Then
\begin{equation}\label{bayart ineq}
\Big(\sum_{\mathbf j\in\mathcal J(m,n)}|c_{\mathbf j}|^2\Big)^{1/2}\le 2^{m/2}\int_{\mathbb T^n}| P(w)|dw,
\end{equation}
where $\mathbb T^n$ stands for the $n$-dimensional torus and $dw$ is the normalized Lebesgue measure on  $\mathbb T^n$.
\end{lemma}

Before giving the proof we define the following operator.
For any $w=(w_1 \ldots, w_n) \in \mathbb{T}^n$ and any $ 1 \leq p \leq \infty $ we define the operator
\begin{eqnarray*}
T_w^p : \mathcal{P}(^m \ell_p^n ) & \longrightarrow & \mathcal{P}(^m \ell_p^n) \\
\displaystyle\sum_{ \mathbf j \in \mathcal{J}(m,n)} a_{\mathbf j } z_{\mathbf j}  & \longmapsto & \displaystyle\sum_{ \mathbf j \in \mathcal{J}(m,n)} a_{\mathbf j } z_{\mathbf j} w_{\mathbf j},
\end{eqnarray*}
which clearly has norm one.

We also need a remark whose proof is straightforward from Definition~\ref{defmix}.

\begin{remark} \label{equivincond}
 Let $(P_i)_{i \in \Lambda}$ be a basis for $\mathcal{P}(^m \mathbb{C}^n)$ and  $(P_i')_{i \in \Lambda}$  its dual basis (i.e., $\langle P_i', P_k\rangle = \delta_{i.k}$). For $ 1 \leq q, p \leq \infty$ and $ n,m \in \mathbb{N}$, $ \chi_{p,q}((P_i)_{i \in \Lambda})$ is exactly the best constant $C > 0$ such that
$$ \sum_{i \in \Lambda} \vert \langle P_i',Q\rangle \langle Q',P_i\rangle \vert \leq C \Vert Q  \Vert_{\mathcal{P}(^m \ell_p^n)} \Vert Q'  \Vert_{\mathcal{P}(^m \ell_q^n)'},$$
for every $Q \in \mathcal{P}(^m \mathbb{C}^n)$ and $Q' \in \mathcal{P}(^m \mathbb{C}^n)'$.
\end{remark}

\begin{proof}[Proof of Theorem~\ref{incondicionalidadbasemonomial}]
 Let $(P_i)_{i \in \Lambda}$ be a basis for $\mathcal{P}(^m \mathbb{C}^n)$ and  $(P_i')_{i \in \Lambda}$  its dual basis. Consider $Q \in \mathcal{P}(^m \mathbb{C}^n)$ and $Q' \in \mathcal{P}(^m \mathbb{C}^n)'$. Since $1=\vert \langle z_{\mathbf j}', z_{\mathbf j} \rangle \vert = \vert \sum_{i \in \Lambda} \langle z_{\mathbf j}', P_i \rangle  \langle P_i', z_{\mathbf j} \rangle \vert$, we have

 \begin{align*}
\sum_{\mathbf j \in \mathcal{J}(m,n)} \vert \langle Q', z_{\mathbf j}\rangle &\langle z_{\mathbf j}', Q \rangle \vert   = \sum_{\mathbf j\in \mathcal{J}(m,n)} \vert \langle Q', z_{\mathbf j}\rangle \langle z_{\mathbf j}', Q \rangle \vert \vert \sum_{i \in \Lambda} \langle z_{\mathbf j}', P_i \rangle  \langle P_i', z_{\mathbf j} \rangle \vert \\
& \leq \sum_{i \in \Lambda} \sum_{\mathbf j \in \mathcal{J}(m,n)} \vert \langle Q', z_{\mathbf j}\rangle \langle z_{\mathbf j}', Q \rangle  \langle z_{\mathbf j}', P_i \rangle  \langle P_i', z_{\mathbf j}  \rangle \vert\\
& { \leq} \sum_{i \in \Lambda} \big(\sum_{\mathbf j \in \mathcal{J}(m,n)} \vert \langle Q', z_{\mathbf j}\rangle \langle z_{\mathbf j}', P_i \rangle  \vert^2 \big)^{\frac{1}{2}}  \big(\sum_{\mathbf j \in \mathcal{J}(m,n)} \vert  \langle z_{\mathbf j}', Q \rangle   \langle P_i', z_{\mathbf j}  \rangle\vert^2 \big)^{\frac{1}{2}} \\
& {\leq} \sum_{i \in \Lambda} 2^{m/2} \int_{\mathbb{T}^n} \vert {\sum_{\mathbf j \in \mathcal{J}(m,n)}  \langle Q', z_{\mathbf j}\rangle \langle z_{\mathbf j}', P_i \rangle w_{\mathbf j} }\vert dw \; \cdot \; 2^{m/2} \int_{\mathbb{T}^n} \vert {\sum_{\mathbf j \in \mathcal{J}(m,n)}  \langle z_{\mathbf j}', Q\rangle \langle P_i', z_{\mathbf j} \rangle \tilde{w}_{\mathbf j} }\vert d\tilde{w} \\
& = \sum_{i \in \Lambda} 2^m \int_{\mathbb{T}^n} \vert {\langle (T_w^q)^*(Q'), P_i\rangle}\vert dw \; \cdot \; \int_{\mathbb{T}^n} \vert {\langle P_i', T_{\tilde{w}}^p(Q) \rangle}\vert d\tilde{w} \\
& = 2^m \int_{\mathbb{T}^n \times \mathbb{T}^n} \sum_{i \in \Lambda} \vert \langle (T_w^q)^*(Q'), P_i\rangle \langle P_i', T_{\tilde{w}}^p(Q) \rangle \vert dw d\tilde{w} \\
& \leq 2^m \int_{\mathbb{T}^n \times \mathbb{T}^n} \chi_{p,q}((P_i)_{i \in \Lambda}) \Vert  (T_w^q)^*(Q') \Vert_{\mathcal{P}(^m \ell_q^n)'} \Vert T_{\tilde{w}}^p(Q)  \Vert_{\mathcal{P}(^m \ell_p^n)} dw d\tilde{w} \\
& \leq 2^m \chi_{p,q}((P_i)_{i \in \Lambda}) \Vert  Q' \Vert_{\mathcal{P}(^m \ell_q^n)'} \Vert Q  \Vert_{\mathcal{P}(^m \ell_p^n)},
 \end{align*}
where we applied Cauchy-Schwarz for the second inequality, Bayart's inequality \eqref{bayart ineq} for the third one and Remark~\ref{equivincond} for the basis $(P_i)$ for the next to last inequality.
 Using Remark~\ref{equivincond} again but for the monomial basis $(z_{\mathbf j})_{ \mathbf j\in \mathcal{J}(m,n)}$ we have that
$$  \chi_{p,q}\big((z_{\mathbf j})_{\mathbf j \in \mathcal{J}(m,n)}\big) \leq 2^m \chi_{p,q}((P_i)_{i \in \Lambda}).$$
Since $(P_i)_{i \in \Lambda}$ is an arbitrary basis of $\mathcal{P}(^m \mathbb{C}^n)$ we have
$$  \chi_{p,q}(\mathcal{P}(^m \mathbb{C}^n)) \leq \chi_{p,q}\big((z_{\mathbf j})_{\mathbf j \in \mathcal{J}(m,n)}\big) \leq 2^m \chi_{p,q}(\mathcal{P}(^m \mathbb{C}^n)),$$
which concludes the proof.
\end{proof}

We now present some estimates for the asymptotic behavior of the mixed-$(p,q)$ unconditional constant of $\mathcal P(^m\mathbb C^n)$. Note that in the case $q=p$ we recover the results from \cite{defant2001unconditional}.
\begin{theorem}\label{cte incond mon}
\[
\begin{cases}  \; \chi_{p,q}(\mathcal{P}(^m \mathbb{C}^n)) \sim 1 & \text{ for } (I): \; [\frac{1}{p} + \frac{m-1}{2m} \leq \frac{1}{q} \wedge \frac{1}{p} \leq \frac{1}{2} ] \text{ or } [ \frac{m-1}{m} + \frac{1}{mp} < \frac{1}{q} \wedge \frac{1}{2} \leq \frac{1}{p} ], \\
\;\chi_{p,q}(\mathcal{P}(^m \mathbb{C}^n)) \sim n^{m (\frac{1}{p}-\frac{1}{q}+\frac{1}{2}) - \frac{1}{2}} & \text{ for } (II) \; \; [ \frac{1}{p} + \frac{m-1}{2m} \geq \frac{1}{q} \wedge \frac{1}{p} \leq \frac{1}{2}  ] ,\\
\; \chi_{p,q}(\mathcal{P}(^m \mathbb{C}^n)) \sim n^{(m-1)(1-\frac{1}{q}) + \frac{1}{p}  -\frac{1}{q}} & \text{ for } (III) \; : \; [ \frac{1}{p} \geq \frac{1}{q} \; \wedge \; \frac{1}{2} \leq \frac{1}{p}  ], \\
\; \chi_{p,q}(\mathcal{P}(^m \mathbb{C}^n)) \sim_{\varepsilon}  n^{(m-1)(1-\frac{1}{q}) + \frac{1}{p}  -\frac{1}{q}} & \text{ for } (III') \; : \; [ 1 - \frac{1}{m} + \frac{1}{mp} \geq \frac{1}{q}\ge\frac1{p} \; \wedge \; \frac{1}{2} < \frac{1}{p} <1 ]. \\
\end{cases}
\]
where $\chi_{p,q}(\mathcal{P}(^m \mathbb{C}^n)) \sim_{\varepsilon} n^{(m-1)(1-\frac{1}{q}) + \frac{1}{p}  -\frac{1}{q}}$ means that
$$  n^{(m-1)(1-\frac{1}{q}) + \frac{1}{p}  -\frac{1}{q}} \ll \chi_{p,q}(\mathcal{P}(^m \mathbb{C}^n)) \ll n^{(m-1)(1-\frac{1}{q}) + \frac{1}{p}  -\frac{1}{q} + \varepsilon},$$
for every $\varepsilon >0$.

Moreover, in $(III')$ for every $\lambda > \frac{1}{p}$ we have
$$  n^{(m-1)(1-\frac{1}{q}) + \frac{1}{p}  -\frac{1}{q}} \ll \chi_{p,q}(\mathcal{P}(^m \mathbb{C}^n)) \ll \log(n)^{m(\frac{1}{q} - \frac{1}{p}) + (\lambda + \frac{1}{p} ) \frac{m^2}{m-1} \frac{p-q}{q(p-1)}}   n^{(m-1)(1-\frac{1}{q}) + \frac{1}{p}  -\frac{1}{q}}.$$
\end{theorem}

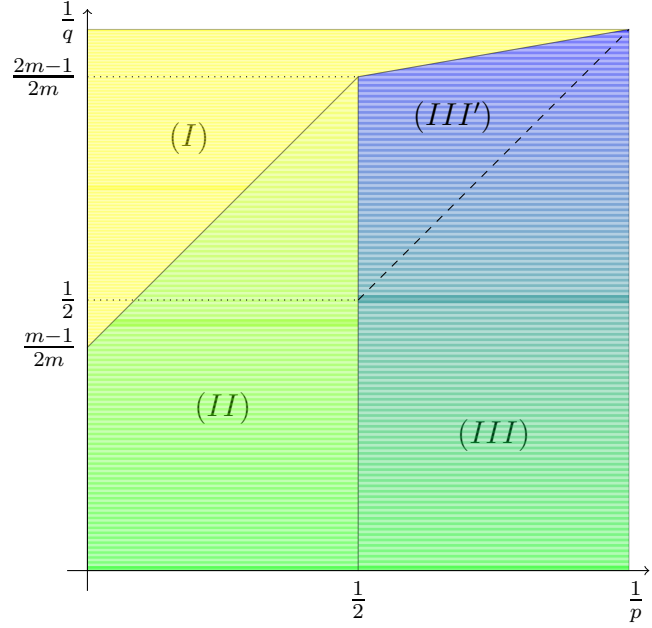
\begin{figure} \label{Figura incondicionalidad}
\begin{center}
\begin{tikzpicture}[scale=0.9]

	\draw (1.5,6.4) node {$(I)$};
  \path[draw, shade, top color=yellow, bottom color=yellow, opacity=.3]
     (4,7.3) node[below] {$ $}  -- (8, 8) -- (0, 8) -- (0,3.3) -- cycle;

    \draw (2,2.4) node {$(II)$};
  \path[draw, shade, top color=yellow, bottom color=green, opacity=.3]
     (0,0) node[below] {$ $}  -- (0, 3.3) -- (4,7.3) -- (4,0) -- cycle;

	\draw (6,2) node {$(III)$};
  \path[draw, shade, top color=blue, bottom color=green, opacity=.3]
    (4,7.3) node[below] {$ $} -- (8,8) -- (8,0) node[below] {$ $}
     -- (4, 0) -- cycle;

\draw[dashed] (4,4) -- (8,8);
	\draw (5.4,6.7) node {$(III')$};

\draw[dotted] (0,4) -- (4,4); 
\draw[dotted] (0,7.3) -- (4,7.3); 

\draw (4,0) node[below] {$\frac12$};
\draw (0,4) node[left] {$\frac12$};
\draw (0,7.3) node[left] {$\frac{2m-1}{2m}$};
\draw (0,3.3) node[left] {$\frac{m-1}{2m}$};

\draw (8.1, 0) node[below] {$\frac1{p} $};
  \draw[->] (-0.3,0) -- (8.3, 0);
\draw (0,8.1) node[left] {$\frac1{q} $};
  \draw[->] (0,-0.3) -- (0, 8.3);
 \end{tikzpicture}
\end{center}

\caption{Graphical overview of the mixed unconditional constant described in Theorem~\ref{cte incond mon}.}

\end{figure}

To prove the theorem we need a lemma and also to recall some results on monomial convergence.
\begin{lemma}\label{ineq chi}
Let $1\le q,p\le\infty$, then we have
\begin{equation}
\chi_{p,q}(\mathcal{P}(^m \mathbb{C}^n)) \ll B_{r,q}^m(n) A_{p,r}^m(n) \; \: \mbox{ for every } \;\; 1 \leq r \leq \infty.
\end{equation}
\end{lemma}

\begin{proof}
Let $P(z) = \displaystyle\sum_{\mathbf j\in \mathcal{J}(m,n)} c_{\mathbf j} z_{\mathbf j} $ be an $m$-homogeneous polynomial in $n$ variables and  $(\theta_{\mathbf j})_{\mathbf j \in \mathcal J(m,n)}$ be a sequence of complex numbers of modulus one, then
$$ \Big\| \displaystyle\sum_{\mathbf j\in \mathcal{J}(m,n) } \theta_{\mathbf j} c_{\mathbf j} z_{\mathbf j} \Big\|_{\mathcal{P}(^m \ell_q^n)} \leq B_{r,q}^m(n) \Big( \displaystyle\sum_{\mathbf j\in \mathcal{J}(m,n)} |c_{\mathbf j}|^r \Big)^{\frac{1}{r}} \leq  B_{r,q}^m(n)  A_{q,r}^m(n) \Big\| \displaystyle\sum_{\mathbf j\in \mathcal{J}(m,n) } c_{\mathbf j} z_{\mathbf j} \Big\|_{\mathcal{P}(^m \ell_p^n)}, $$ for every $ 1 \leq r \leq \infty$. Thus $\chi_{p,q}((z_{\mathbf j})_{\mathbf j\in \mathcal{J}(m,n)}) \leq B_{r,q}^m(n) A_{p,r}^m(n)$ and the result follows from Theorem~\ref{incondicionalidadbasemonomial}.
\end{proof}

The set of monomial convergence for $m$-homogeneous polynomials over the domain $\ell_p$, denoted by  $mon (\mathcal P(^m \ell_p))$, is defined as
\begin{equation}
 mon((\mathcal P(^m \ell_p)) := \{ z \in \ell_p: \sum_{\alpha \in \N_0^{\N}} \vert c_\alpha(P) z^{\alpha} \vert < \infty, \mbox{ for all } P \in \mathcal P(^m \ell_p))\}
\end{equation}

See \cite{defant2009domains, bayart2015monomial} and the references therein for what is known about these sets. There is a strong relation between monomial convergence and mixed unconditionality.
Let $X,Y$ be Banach sequence spaces. We denote by $X_n$ and $Y_n$ the $n$-dimensional subspaces given by the span of the first $n$ canonical vectors in $X$ and $Y$ respectively. The following result is an immediate consequence of \cite[Theorem 5.2]{defant2009domains}.
\begin{theorem}\label{relacion mon}
 Let $X,Y$ be Banach sequence spaces. Then the following are equivalent
\begin{itemize}
 \item[i)] $ Y\subset mon(\mathcal P(^m X))$,
 \item[ii)] there exists $C=C(m)>0$ such that for every $m$-homogeneous polynomial in $n$ variables  $P(z) = \displaystyle\sum_{\mathbf j\in \mathcal{J}(m,n)} c_{\mathbf j} z_{\mathbf j} $, and every $u\in B_{Y_n}$
 $$\displaystyle\sum_{\mathbf j\in \mathcal{J}(m,n)} |c_{\mathbf j} u_{\mathbf j}| \le C\|P\|_{\mathcal P(^m X_n)}.$$
 \end{itemize}
 In particular,
 \begin{equation}\label{relacion mon y chi}
\ell_q \subset mon(\mathcal P(^m \ell_p))\; \mbox{ if and only if  } \; \chi_{p,q}(\mathcal{P}(^m \mathbb{C}^n)) \sim 1.
 \end{equation}
 \end{theorem}

 Let $\mathbf p=(p_j)_{j\in\mathbb N}$ be the  sequence of prime numbers. For a real function $f$, we define $f(\mathbf p):=(f(p_j))_{j\in\mathbb N}$, and we denote by $f(\mathbf p)\cdot\ell_p$ the set $\{(f(p_j)x_j)_{j\in\mathbb N}:\, x\in\ell_p\}$.
 \begin{lemma} \label{inclusion mon}
  Let  $\frac{1}{2} \leq \frac{1}{p}$ and $\frac{1}{q_m}= 1 - \frac{1}{m} + \frac{1}{mp}$. For $\lambda>\frac1{p}$, we define the sequence $w_{\lambda}:=(w_{\lambda}(j))_{j\in\mathbb N}$ where $w_{\lambda}(j)=\log(p_j)^{q_m \lambda + 1}$ and define the Lorentz Banach sequence space
  $$
d(w_{\lambda},q_m)=\Big\{y=(y_j)_{j\in\mathbb N}:\, \|y\|_{d(w_{\lambda},q_m)}:=\left(\sum_{j}|y_j^*|^{q_m} w_{\lambda}(j) \right)^{1/q_m} < \infty \Big\},
$$
where $y^*$ denotes the decreasing rearrangement of $y$. Then, $$d(w_{\lambda},q_m)\subset mon(\mathcal P(^m \ell_p)).$$
 \end{lemma}
\begin{proof}
Note that $y \in mon(\Pp(^m \ell_p))$ if and only if its decreasing rearrangement $y^*$ belongs to $ mon(\Pp(^m \ell_p))$.
 In \cite[Theorem 5.3]{bayart2015monomial} it was proved that  for any $\varepsilon > \frac{1}{p}$ we have,
 $$
 \frac{1}{\mathbf p^{\sigma_m} \log(\mathbf p)^\varepsilon} \cdot \ell_p \subset mon(\Pp(^m \ell_p)),
 $$
 where $\sigma_m = \left( \frac{m-1}{m} \right) \left(1 - \frac{1}{p} \right) $.
 We will show that for every $y \in d(w_{\lambda},q_m)$ we have that  $y^* \in \frac{1}{\mathbf p^{\sigma_m} \log(\mathbf p)^{\varepsilon_0}} \cdot \ell_p$ for some $\varepsilon_0 > \frac{1}{p}$.

%

%

Notice that
\begin{equation} \label{cota coef lorentz}
 \log(p_n)^{\lambda q_m+1} (y_n^*)^{q_m} \ll \frac{1}{n},
\end{equation}
indeed, if not there is a subsequence $(n_k)$ and constants $C_k\to\infty$ such that 
$$ \log(p_{n_k})^{\lambda q_m+1} (y_{n_k}^*)^{q_m} \ge  C_k\frac{1}{{n_k}}.$$ 
By the Prime Number Theorem ($p_n \sim n \log(n)$) and some elementary computations, we have that for sufficiently large $k$, $\log(p_j)/\log(p_{n_k})\ge 1/2$ if $j\ge n_k/2$. Thus, 
\begin{align*}
\sum_{j=1}^{n_k}\log(p_j)^{\lambda q_m+1} (y_j^*)^{q_m} &  \ge (y_{n_k}^*)^{q_m} \sum_{n_k/2\le j \le n_k}\log(p_j)^{\lambda q_m+1} \\
& \ge \frac12 (y_{n_k}^*)^{q_m} \log(p_{n_k})^{\lambda q_m+1} \cdot \frac{n_k}{2}\\
& \ge \frac{C_k}{4} \to \infty,
\end{align*} 
which is a contradiction since $y\in d(w_{\lambda},q_m)$. Therefore, by \eqref{cota coef lorentz},
$$  y_n^* \ll \frac{1}{n^{\frac{1}{q_m}}} \; \frac1{log(p_n)^{\frac{1}{q_m}+\lambda}}\le \frac{1}{n^{\frac{1}{q_m}}} \; \frac1{log(n)^{\frac{1}{q_m}+\lambda}}.$$
If we take $\varepsilon_0, \tilde \lambda$ such that $\frac{1}{p}<\varepsilon_0<\tilde \lambda<\lambda$, and use the Prime Number Theorem  we have
\begin{align*}
p_n^{\sigma_m} \log( p_n)^{\varepsilon_0} y_n^* &\ll \frac{p_n^{\sigma_m}\log(p_n)^{\varepsilon_0}}{n^{\frac{1}{q_m}}  \log(n)^{\frac{1}{q_m} + \lambda }} \\
& \ll \frac{n^{\sigma_m} \log(n)^{\sigma_m+\tilde\lambda}}{n^{\frac{1}{q_m}} \log(n)^{\frac{1}{q_m} + \lambda}}  \\
& = n^{- \frac{1}{p}} \log(n)^{ -\frac{1}{p}+\tilde\lambda - \lambda}.
\end{align*}
Since the sequence $ (n^{- \frac{1}{p}} \log(n)^{ -\frac{1}{p}+\tilde\lambda-\lambda})_{n \ge 1}$ is in $\ell_p$, we conclude the proof.
\end{proof}

%

\begin{lemma} \label{inc log}
Let $\frac{1}{2} \leq \frac{1}{p}$ and $ \frac{1}{q_m}= 1 - \frac{1}{m} + \frac{1}{mp}  $. Then if $\lambda>\frac1{p}$,
 $$
 \chi_{p,q_m}(\mathcal{P}(^m \mathbb{C}^n)) \ll  \log(n)^{ m ( \lambda + \frac{1}{q_m})}.
$$
\end{lemma}

\begin{proof}
 Let $\lambda> \tilde \lambda > \frac1{p}$. If $u \in B_{\ell_{q_m}^n}$ then, by the Prime Number Theorem, we know that
 $$ \Vert u \Vert_{d(w_{\tilde\lambda},q_m)_n} \leq \log(p_n)^{\tilde \lambda + \frac{1}{q_m}} \ll \log(n)^{\lambda+ \frac{1}{q_m}}.$$
 For every $m$-homogeneous polynomial in $n$ variables  $P(z) = \displaystyle\sum_{\mathbf j\in \mathcal{J}(m,n)} c_{\mathbf j} z_{\mathbf j} $ and every $u\in B_{\ell_{q_m}^n}$,  by Lemma \ref{inclusion mon} and Theorem \ref{relacion mon}, we have that
  \begin{align*}
\displaystyle\sum_{\mathbf j\in \mathcal{J}(m,n)} |c_{\mathbf j} u_{\mathbf j}| & \le C \Vert u \Vert_{d(w_{\tilde \lambda},q_m)_n}^m \|P\|_{\mathcal P(^m \ell_p^n)} \\
& \ll \log(n)^{m ( \lambda + \frac{1}{q_m})} \|P\|_{\mathcal P(^m \ell_p^n)},
 \end{align*}
and from this inequality it is easy  to conclude the proof.
 \end{proof}

We now prove Theorem~\ref{cte incond mon}.

\begin{proof}[Proof of Theorem~\ref{cte incond mon}]
The proof is divided in cases.

$\bullet (I):$
Let $p \geq 2$. By \cite[Example 4.6]{defant2009domains} we know that
$\ell_{q_m} \subset mon(\mathcal P(^m \ell_p))$ where
$$ \frac{1}{q_m} = \frac{1}{p} + \frac{m-1}{2m}.$$
On the hand, if $p \leq 2$, by \cite[Theorem 5.1]{bayart2015monomial} we know that
$\ell_{q_m - \varepsilon} \subset mon(\mathcal P(^m \ell_p))$ for every $\varepsilon >0$ where
$$ \frac{1}{q_m} = \frac{m-1}{m} + \frac{1}{mp}.$$

Therefore, by the statement in \eqref{relacion mon y chi} and monotonicity we known that $\chi_{p,q}(\mathcal{P}(^m \mathbb{C}^n)) \sim 1$ in region
$$(I): \; [\frac{1}{p} + \frac{m-1}{2m} \leq \frac{1}{q} \wedge \frac{1}{p} \leq \frac{1}{2} ] \text{ or } [ \frac{m-1}{m} + \frac{1}{mp} < \frac{1}{q} \wedge \frac{1}{2} \leq \frac{1}{p} ].$$


$\bullet (II):$ We know by $(I)$ that $\chi_{p,q_m}(\mathcal{P}(^m \mathbb{C}^n))\sim 1$, for $\frac1{q_m}=\frac1{p}+\frac{m-1}{2m}$.
We now  estimate $\chi_{p,\infty}(\mathcal{P}(^m \mathbb{C}^n))$ for $0\le \frac1{p}\le \frac12$.
 Take $ {r}=1 $.
By  Proposition~\ref{asint B} and Theorem~\ref{teorema con A} $(C)$ we have
$$
B_{r,\infty}^m(n) \sim 1 \; , \; A_{p,r}^m(n) \sim n^{m(\frac{1}{p}+\frac{1}{2})-\frac{1}{2}}.
$$
Using Lemma~\ref{ineq chi}
$$
\chi_{p,\infty}(\mathcal{P}(^m \mathbb{C}^n)) \ll n^{m(\frac{1}{p}+\frac{1}{2})-\frac{1}{2}}.
$$

  Take a polynomial $P\in\mathcal{P}(^m \mathbb{C}^n)$, $P= \displaystyle\sum_{\mathbf j\in \mathcal{J}(m,n) }  c_{\mathbf j} z_{\mathbf j}$ with $\|P\|_{\mathcal P(^m\ell^n_p)}=1$ and take signs $(\theta_{\mathbf j})_{\mathbf j\in \mathcal{J}(m,n)}$. Therefore since
\begin{align*}
\Big\|\sum_{\mathbf j\in \mathcal{J}(m,n) } \theta_{\mathbf j} c_{\mathbf j} z_{\mathbf j}\Big\|_{\mathcal P(^m\ell^n_{q_m})}&\le \chi_{p,q_m}(\mathcal{P}(^m \mathbb{C}^n))\sim 1,\\
\Big\|\sum_{\mathbf j\in \mathcal{J}(m,n) } \theta_{\mathbf j} c_{\mathbf j} z_{\mathbf j}\Big\|_{\mathcal P(^m\ell^n_{\infty})}&\le  \chi_{p,\infty}(\mathcal{P}(^m \mathbb{C}^n))\ll n^{m(\frac{1}{p}+\frac{1}{2})-\frac{1}{2}},
\end{align*}
we have, by \eqref{P vs T unif} and The Multilinear Interpolation Theorem (see \cite[Section 4.4]{BerLof76}) that for $\theta \in (0,1)$ and $\frac1{q}=\frac{\theta}{q_m}+\frac{1-\theta}{\infty},$
$$
\|\sum_{\mathbf j\in \mathcal{J}(m,n) } \theta_{\mathbf j} c_{\mathbf j} z_{\mathbf j}\|_{\mathcal P(^m\ell^n_{q})}\ll n^{(1-\theta)[m(\frac{1}{p}+\frac{1}{2})-\frac{1}{2}]}=n^{m (\frac{1}{p}-\frac{1}{q}+\frac{1}{2}) - \frac{1}{2}}.
$$

For the lower bound let $P(z) = \displaystyle\sum_{\alpha \in \Lambda(m,n)} \varepsilon_\alpha z^\alpha $ be a unimodular polynomial as in \eqref{polinomios de Bayart} with $p\ge 2$, then if $w=(\frac1{n^{1/q}},\dots,\frac1{n^{1/q}})\in S_{\ell_q}$, we have
$$
\chi_{p,q}(\mathcal{P}(^m \mathbb{C}^n)) \gg \frac{\| \displaystyle\sum_{\alpha \in \Lambda(m,n)} z^\alpha \|_{\mathcal{P}(^m \ell_q^n)}}{\| \displaystyle\sum_{\alpha \in \Lambda(m,n)} \varepsilon_\alpha z^\alpha \|_{\mathcal{P}(^m \ell_p^n)}} \gg \frac{| \displaystyle\sum_{\alpha \in \Lambda(m,n)} w^\alpha |}{n^{m(\frac{1}{2}-\frac{1}{p}) + \frac{1}{2}}} \gg \frac{n^{m(1-\frac{1}{q})}}{n^{m(\frac{1}{2}-\frac{1}{p}) + \frac{1}{2}}} = n^{m(\frac{1}{p}-\frac{1}{q}+\frac{1}{2}) - \frac{1}{2}}.$$

$\bullet (III):$ For $[ \frac{1}{2} \leq \frac{1}{p} \; \wedge \; \frac{1}{q} \leq \frac{1}{p} ]$ let us take $ \frac{1}{r} = 1 - \frac{1}{q}$. Note that $\frac1{r}\ge 1 - \frac{1}{p}$. Then by Proposition~\ref{asint B} and Theorem~\ref{teorema con A} $(D)$,
$$
B_{r,q}^m(n) \sim 1 \; , \; A_{p,r}^m(n) \sim n^{(m-1)(1-\frac{1}{q})+\frac{1}{p}-\frac1{q}},
$$
and therefore
$$
\chi_{p,q}(\mathcal{P}(^m \mathbb{C}^n)) \ll n^{(m-1)(1 - \frac{1}{q}) + \frac{1}{p} - \frac{1}{q} }.
$$
For the lower bound let $P(z) = \displaystyle\sum_{\alpha \in \Lambda(m,n)} \varepsilon_\alpha z^\alpha $ be a unimodular polynomial as in \eqref{polinomios de Bayart} with $ 1 \leq p \leq 2$, we have
\begin{equation} \label{Bayart incond}
\chi_{p,q}(\mathcal{P}(^m \mathbb{C}^n))\gg \frac{\| \displaystyle\sum_{\alpha \in \Lambda(m,n)} z^\alpha \|_{\mathcal{P}(^m \ell_q^n)}}{\| \displaystyle\sum_{\alpha \in \Lambda(m,n)} \varepsilon_\alpha z^\alpha \|_{\mathcal{P}(^m \ell_p^n)}} \gg \frac{n^{m(1-\frac{1}{q})}}{n^{1-\frac{1}{p}}} = n^{(m-1)(1-\frac{1}{q}) + \frac{1}{p} - \frac{1}{q}}.
\end{equation}
$\bullet (III'):$ Let $\frac{1}{2} < \frac{1}{p} < 1$.
For every $\lambda > \frac{1}{p}$ we know by Lemma~\ref{inc log}  that $\chi_{p,q_m}(\mathcal{P}(^m \mathbb{C}^n))\ll \log(n)^{ m (\lambda+\frac1{q_m})}$,  where  $\frac1{q_m}=1 - \frac{1}{m}+\frac{1}{mp}$.
On the other hand, by $(III)$ we know that $\chi_{p,p}(\mathcal{P}(^m \mathbb{C}^n))\sim n^{(m-1)(1-\frac1{p})}$.

As in $(II)$ we will use the The Multilinear Interpolation Theorem: take a polynomial $P\in\mathcal{P}(^m \mathbb{C}^n)$, $P= \displaystyle\sum_{\mathbf j\in \mathcal{J}(m,n) }  c_{\mathbf j} z_{\mathbf j}$ with $\|P\|_{\mathcal P(^m\ell^n_p)}=1$ and take signs $(\theta_{\mathbf j})_{\mathbf j\in \mathcal{J}(m,n)}$. Therefore since
\begin{align*}
\Big\|\sum_{\mathbf j\in \mathcal{J}(m,n) } \theta_{\mathbf j} c_{\mathbf j} z_{\mathbf j}\Big\|_{\mathcal P(^m\ell^n_{q_m})}&\le \chi_{p,q_m}(\mathcal{P}(^m \mathbb{C}^n))\ll \log(n)^{m( \lambda + \frac{1}{q_m}) },\\
\Big\|\sum_{\mathbf j\in \mathcal{J}(m,n) } \theta_{\mathbf j} c_{\mathbf j} z_{\mathbf j}\Big\|_{\mathcal P(^m\ell^n_p)}&\le  \chi_{p,p}(\mathcal{P}(^m \mathbb{C}^n))\sim n^{(m-1)(1-\frac1{p})},
\end{align*}
we have, by \eqref{P vs T unif} and \cite[Section 4.4]{BerLof76}, that for $\theta \in (0,1)$ and $\frac1{q}=\frac{\theta}{q_m}+\frac{1-\theta}{p},$
\begin{align*}
\|\sum_{\mathbf j\in \mathcal{J}(m,n) } \theta_{\mathbf j} c_{\mathbf j} z_{\mathbf j}\|_{\mathcal P(^m\ell^n_{q})} 
& \ll \log(n)^{\theta m ( \lambda + \frac{1}{q_m})} n^{(1-\theta)[(m-1)(1-\frac1{p})]}  \\
& = \log(n)^{\theta m ( \lambda + \frac{1}{q_m})} \;  n^{(m-1)(1-\frac{1}{q}) + \frac{1}{p}  -\frac{1}{q}}, 
\end{align*}
with $\theta=\frac{m}{m-1}p'(\frac1{q}-\frac1{p})$.
This concludes the upper bound for the region $(III'): [ 1 - \frac{1}{m} + \frac{1}{mp} \geq \frac{1}{q}\ge\frac1{p} \; \wedge \; \frac{1}{2} < \frac{1}{p} <1 ]$.

The lower bound is exactly as in \eqref{Bayart incond}.
\end{proof}


\section{Some applications to the multivariable von Neumann's inequality} \label{applicvn}
A classical inequality in operator theory, due to  von Neumann \cite{vNe51}, asserts that if $T$ is  a linear
contraction on a complex Hilbert space $\mathcal H$ (i.e., its operator norm is less than or equal to one)
then
$$
\|p(T)\|_{\mathcal L(\mathcal H)}\le  \sup\{|p(z)| : z\in \mathbb C, \, |z|\le1 \},
$$
for every polynomial $p$ in one (complex) variable.

Using dilation theory (see \cite{SzN74}), for polynomials in two commuting  contractions, Ando \cite{And63} exhibited
an analogue inequality. However Varopoulos \cite{Var74} showed that von Neumann's
inequality cannot be extended for three or more commuting contractions.

It is an open problem of great interest in operator theory (see for example \cite{Ble01,Pis01}) to determine whether
there exists a constant $K (n)$ that adjusts von Neumann's inequality. More precisely, it is unknown whether
or not there exists a constant $K(n)$ such that
\begin{equation} \label{Pisier problem}
\|p(T_1,\dots,T_n)\|_{\mathcal L(\mathcal
H)}\le K(n) \; \sup\{ |p(z_1, \dots, z_n)| : |z_i| \leq 1 \},
\end{equation}
for every polynomial $p$ in $n$ variables and every $n$-tuple $(T_1, \dots, T_n)$ of commuting contractions in
$\mathcal L(\mathcal{H})$.

Dixon \cite{dixon1976neumann} studied the multivariable von Neumann's inequality restricted to homogeneous polynomials and \cite{mantero1979banach} studied some variations of this problem.
One of them is to determine the asymptotic behavior of the best possible constant $c(n)=c_{m,p,q}(n)$ such that
\begin{align*}
 \|P(T_1,\dots,T_n)\|_{\mathcal L(\mathcal H)}\le c(n)\|P\|_{\mathcal P(^m\ell_q^n)},
\end{align*}
for every $n$-tuple $T_1,\dots,T_n$  of commuting operators on a Hilbert space satisfying
\begin{align}
\sum_{i=1}^n\|T_i\|_{\mathcal L(\mathcal H)}^p\le 1, \label{Ip}
\end{align}
and any $m$-homogeneous polynomial on $n$ variables, $P$. Some lower bounds were proven there and also some upper bounds were given for the case $p=q$. We will apply the results of Section \ref{Mainresult} to show upper bounds for $c(n)$ for any $1\le p,q\le\infty$.

Recall that given a bilinear form $a: X_1 \times X_2 \to \mathbb{C}$ its uniform norm is $$\Vert a \Vert_{Bil(X_1 \times X_2)} : = \sup_{(x_1,x_2) \in B_{X_1} \times B_{X_2}} \vert a(x_1,x_2) \vert.$$
We need the following  lemma from \cite{mantero1979banach} which is an easy consequence of the Grothendieck inequality. We prove it for the sake of completeness.
\begin{lemma}
 For  $i=1,\dots,N$, $j=1,\dots,M$ let $x_i,y_j$ be vectors in some Hilbert space $\mathcal H$ such the $\sum_{i=1}^N \|x_i\|_{\mathcal H}^p\le1$ and $\sum_{j=1}^M  \|y_j\|_{\mathcal H}^p\le1$, and let $(a_{i,j})_{i,j}\in\mathbb C^{N \times M}$. Then
 \begin{align*}
  \left|\sum_{i,j}a_{ij}\langle x_i,y_j\rangle\right|\le K_G\|a\|_{Bil(\ell_p^N\times\ell_p^M)},
 \end{align*}
where $K_G$ denotes the Grothendieck constant and $a$ is the bilinear form on $\mathbb C^N\times\mathbb C^M$ whose coefficients are the $a_{ij}$'s.
\end{lemma}
\begin{proof}
 \begin{align*}
  \left|\sum_{i,j}a_{i,j}\langle x_i,y_j\rangle\right| & \le \left|\sum_{i,j}a_{i,j}\|x_i\|_{\mathcal H}\|y_j\|_{\mathcal H}\langle \frac{x_i}{\|x_i\|_{\mathcal H}},\frac{y_j}{\|y_j\|_{\mathcal H}}\rangle\right|  \\
  & \le K_G \sup\{\sum_{i,j}a_{i,j}\|x_i\|_{\mathcal H}\|y_j\|_{\mathcal H}\beta_i\gamma_j \; :\; \beta\in B_{\ell_\infty^N},\gamma\in B_{\ell_\infty^M}\} \\
  & \le K_G\|a\|_{Bil(\ell_p^N\times\ell_p^M)}.
 \end{align*}
\end{proof}

\begin{proposition}
 Let $T_1,\dots,T_n$  be commuting operators on a Hilbert space $\mathcal H$ satisfying (\ref{Ip}) and $P\in\mathcal P(^m\mathbb C^n)$. Then
 \begin{align*}
   \|P(T_1,\dots,T_n)\|_{\mathcal L(\mathcal H)} \le C A^{m-1}_{q,p'}(n)\|P\|_{\mathcal P(^m\ell_q^n)},
 \end{align*}
 where $C$ is constant independent of $n$.
\end{proposition}
\begin{proof}

Let $a_{\mathbf i}$, $\mathbf i \in\mathcal M(m,n)$ be the coefficients of the symmetric $m$-linear  form $a$ associated to $P$, and let $x,y$ be unit vectors in $\mathcal H$. Note that we may also view $a$ as a bilinear form on $\mathbb C^{n^{m-1}}\times\mathbb C^n$, then by the previous lemma,
\begin{align*}
  \left|\sum_{\mathbf i \in\mathcal M(m,n)}a_{\mathbf i}\langle T_{i_1}\dots T_{i_m}x,y\rangle\right| & = \left|\sum_{(\mathbf i,j) \in\mathcal M(m-1,n)\times\{1,\dots,n\}} a_{(\mathbf i,j)}\langle T_{i_1}\dots T_{i_{m-1}}x,T_j^ *y\rangle\right|  \\
  & \le K_G \|a\|_{Bil(\ell_p^{n^{m-1}}\times\ell_p^n)} = K_G \sup_{\beta\in B_{\ell_p^n}}\left(\sum_{\mathbf i \in\mathcal M(m-1,n)}\left|\sum_{j=1}^na_{(\mathbf i,j)}\beta_j\right|^{p'}\right)^{1/{p'}}.
 \end{align*}
Note that $\sum_{j=1}^na_{(\mathbf i,j)}\beta_j$ are the coefficients of the $(m-1)$-linear form $a_\beta$ which is obtained by fixing one variable of $a$ at $\beta$, that is, $a_\beta(v_1,\dots,v_{m-1})=a(\beta,v_1,\dots,v_{m-1})$. Note also that the $p'$-norm of the coefficients of $a_\beta$ is less than or equal to the $p'$-norm of the coefficients of the associated polynomial $P_\beta$.
 Then, since $\|P_\beta\|_{\mathcal P(^{m-1}\ell_q^n)}\le e\|P\|_{\mathcal P(^{m}\ell_q^n)}$ (see for example \cite{harris1972bounds}),  taking supremum over $x,y\in B_{\mathcal H}$ we have
\begin{align*}
 \|P(T_1,\dots,T_n)\|_{\mathcal{L}(\mathcal{H})} \le K_G \sup_{\beta\in B_{\ell_p^n}} A^{m-1}_{q,p'}(n)\|P_\beta\|_{\mathcal P(^{m-1}\ell_q^n)}\le K_GeA^{m-1}_{q,p'}(n)\|P\|_{\mathcal P(^{m}\ell_q^n)}.
 \end{align*}
\end{proof}

\begin{remark}
 Taking $p=q$ and using Theorem~\ref{teorema con A} we recover the inequality proved in \cite{mantero1979banach}, that is, $c(n)\ll n^{\frac{m-2}{p'}}$ if $p\le2$ and $c(n)\ll n^{\frac{m-2}{2}}$ if $p\ge2$.

\end{remark}

We also have the following corollary.

\begin{corollary}

Let $T_1,\dots,T_n$  be commuting operators on a Hilbert space $\mathcal H$ satisfying (\ref{Ip}). If  $ [\frac{1}{2}\le\frac1{p'}\le \frac{m}{2(m-1)}-\frac1{q} ] \text{ or } [\frac{1}{p'} \leq \frac{1}{2} \; \wedge \; \frac{m-1}{q} \leq 1 - \frac{1}{p'} ],$ we have that
\begin{align*}
   \|P(T_1,\dots,T_n)\|_{\mathcal{L}(\mathcal{H})} \leq D \|P\|_{\mathcal P(^m\ell_q^n)},
\end{align*}
for every  $m$-homogeneous polynomial $P$, where $D$ is constant independent of $n$.
 \end{corollary}

Another variant studied in \cite{mantero1979banach} is to determine the  best possible constant $d(n)=d_{m,p,q}(n)$ such that
\begin{align*}
 \|P(T_1,\dots,T_n)\|_{\mathcal{L}(\mathcal{H})} \le d(n)\|P\|_{\mathcal P(^m\ell_q^n)},
\end{align*}
for every $m$-homogeneous polynomial in $n$ variables, $P$, and every  $n$-tuple $T_1,\dots,T_n$  of commuting operators on a Hilbert space $\mathcal H$ satisfying
\begin{align}
\left(\sum_{i=1}^n|\langle T_ix,y\rangle|^p\right)^{1/p}\le \|x\|_{\mathcal H} \|y\|_{\mathcal H}, \label{IIp}
\end{align}
for any vectors $x,y\in \mathcal H$. Note that (\ref{IIp}) is equivalent to $\|\sum_{i=1}^nT_i\beta_i\|_{\mathcal{L}(\mathcal{H})}\le\|\beta\|_{p'}$, for every $\beta\in\mathbb C^n$.
\begin{lemma}
 Let $T_1,\dots,T_n\in\mathcal L(\mathcal H)$ be operators satisfying (\ref{IIp}), and let $x,y\in \mathcal H$. Then if $Q$ is the $m$-homogeneous polynomial in $n$ variables defined by
 $$
 Q(z)=\sum_{\mathbf i \in\mathcal M(m,n)}\langle T_{i_1}\dots T_{i_m}x,y\rangle z_{i_1}\dots z_{i_m},
 $$
 we have $\displaystyle \|Q\|_{\mathcal P(^m\ell_{p'}^n)}\le\|x\|_{\mathcal{H}} \|y\|_{\mathcal{H}}$.
\end{lemma}
\begin{proof}
 \begin{align*}
 \|Q\|_{\mathcal P(^m\ell_{p'}^n)} & = \sup_{z\in B_{\ell_{p'}^n}}\left|\sum_{\mathbf i \in\mathcal M(m,n)}\langle T_{i_1}\dots T_{i_m}x,y\rangle z_{i_1}\dots z_{i_m}\right| = \sup_{z\in B_{\ell_{p'}^n}}\left|\left\langle \sum_{\mathbf i \in\mathcal M(m,n)}T_{i_1}\dots T_{i_m}z_{i_1}\dots z_{i_m}x,y\right\rangle\right| \\
  & \le \sup_{z\in B_{\ell_{p'}^n}}\left|\left\langle \left(\sum_{l=1}^nz_lT_{l}\right)^mx,y\right\rangle\right|
   \le \sup_{z\in B_{\ell_{p'}^n}}\left\|\sum_{l=1}^nz_lT_{l}\right\|^m \|x\|_{\mathcal{H}} \|y\|_{\mathcal{H}} \le \|x\|_{\mathcal{H}} \|y\|_{\mathcal{H}}.
 \end{align*}

\end{proof}

\begin{proposition}
  Let $T_1,\dots,T_n$  be commuting operators on a Hilbert space $\mathcal H$ satisfying (\ref{IIp}) and $P\in\mathcal P(^m\mathbb C^n)$. Then
 \begin{align*}
   \|P(T_1,\dots,T_n)\|_{\mathcal L(\mathcal H)} \le  A^{m}_{q,r}(n)A^{m}_{p',r'}(n)\|P\|_{\mathcal P(^m\ell_q^n)}.
 \end{align*}
 \end{proposition}
\begin{proof}
 Let $a_{\mathbf i}$ be the coefficients of the symmetric $m$-linear  form $a$ associated to $P$ and $x,y$ unit vectors in $\mathcal H$. Then by the previous lemma and the fact that the $r$-norm of the coefficients of $a$ is less than or equal to the $r$-norm of the coefficients of the associated polynomial $P$, we have
 \begin{align*}
  \left|\sum_{\mathbf i \in\mathcal M(m,n)}a_{\mathbf i}\langle T_{i_1}\dots T_{i_m}x,y\rangle\right| & \le \left(\sum_{\mathbf i \in\mathcal M(m,n)}|a_{\mathbf i}|^r\right)^{1/r}\left(\sum_{\mathbf i \in\mathcal M(m,n)}|\langle T_{i_1}\dots T_{i_m}x,y\rangle|^{r'}\right)^{1/r'} \\
  & \le A^{m}_{q,r}(n)\|P\|_{\mathcal P(^m\ell_q^n)}A^{m}_{p',r'}(n).
 \end{align*}
\end{proof}

\begin{remark}
 Taking $p=q=r'$ and using Theorem~\ref{teorema con A} we recover the inequality proved in \cite[Proposition 20]{mantero1979banach}, that is $d(n)\ll n^{(m-1)(\frac1{p'}+\frac1{2})}$ if $p\le2$ and $d(n)\ll n^{(m-1)(\frac1{p}+\frac1{2})}$ if $p\ge2$. Note also that, in the last proposition,  we have bounds that do not depend on $n$ for some combinations of $p$ and $q$, e.g. for $(p,q)=(1,\infty)$.
\end{remark}

\newcommand{\etalchar}[1]{$^{#1}$}


\end{document}